\documentclass[a4paper,reqno]{amsart}

\usepackage{amsmath, amssymb, amsthm}
\usepackage{enumitem}
\usepackage{hyperref}
\hypersetup{
  colorlinks=false,
  pdfborder={0 0 0}
}
\usepackage{xcolor}

\usepackage[pagewise]{lineno}

\newtheorem{theorem}{Theorem}[section]
\newtheorem{lemma}[theorem]{Lemma}
\newtheorem{proposition}[theorem]{Proposition}
\newtheorem{corollary}[theorem]{Corollary}

\theoremstyle{definition}
\newtheorem{definition}[theorem]{Definition}

\theoremstyle{definition}
\newtheorem{remark}[theorem]{Remark}

\theoremstyle{definition}
\newtheorem{example}[theorem]{Example}

\theoremstyle{definition}
\newtheorem{assumption}[theorem]{Assumption}

\theoremstyle{definition}
\newtheorem{notation}[theorem]{Notation}

\theoremstyle{plain}
\newtheorem{mainresult}{Main Result}

\makeatletter

\@addtoreset{equation}{subsection}
\makeatother

\makeatletter
\renewenvironment{proof}[1][\proofname]{%
  \par\pushQED{\qed}%
  \normalfont%
  \topsep6\p@\@plus6\p@\relax
  \trivlist
  \item[\hskip\labelsep\bfseries #1\@addpunct{.}]%
}{%
  \popQED\endtrivlist\@endpefalse
}
\makeatother

\newcommand{\Poly}{\mathrm{Poly}}

\newcommand{\supp}{\operatorname{supp}}

\newcommand{\U}{\mathbf{U}}
\newcommand{\V}{\mathbf{V}}

\newcommand{\Stat}{\mathsf{Stat}}

\makeatletter
\let\oldtocsection=\tocsection
\renewcommand{\tocsection}[3]{%
  \oldtocsection{#1}{#2}{\textbf{#3}}%
}
\let\oldtocsubsection=\tocsubsection
\renewcommand{\tocsubsection}[3]{%
  \oldtocsubsection{#1}{#2}{\hspace*{0.25em}#3}%
}
\makeatother

\author{Yoshiyuki Endo}
\address{Department of Mathematics, Nagoya University}
\email{endo.yoshiyuki.e2@s.mail.nagoya-u.ac.jp}

\title[Polynomial RSCC and the Probability of Tending to Infinity]{Polynomial Random Dynamical Systems with Complete Connections and the Probability of Tending to Infinity}


\begin{document}

%\linenumbers

\begin{abstract}
We study polynomial random dynamical systems with complete connections on the
Riemann sphere. In this framework, the choice of the next polynomial map is
governed by a state-dependent rule that may have memory. This framework
includes i.i.d. random dynamics and non-i.i.d. Markovian models as special
cases.

For each initial state, we define the probability that the corresponding random
orbit tends to infinity. We prove that this escaping probability is locally
constant on the Fatou set and becomes continuous on the whole Riemann sphere
when all kernel Julia sets are empty. We also formulate the associated
product-space transition operator and use it to obtain continuity with respect
to both the spatial and state variables.

We further introduce stationary-averaged escaping probabilities with respect to
stationary distributions of the induced state chain. Under the same
kernel-emptiness assumption, these averaged escaping probabilities are
continuous. 
We also give a sufficient
condition under which the stationary-averaged escaping probability is
everywhere positive and not identically equal to $1$.

Finally, we present examples illustrating phenomena specific to
this framework,
including reinforcement-induced discontinuity, recovery of continuity under
truncation, and genuinely mixed escaping behavior produced by stationary
averaging.
\end{abstract}

\maketitle

\begin{center}
Graduate School of Mathematics, Nagoya University\\
\texttt{endo.yoshiyuki.e2@s.mail.nagoya-u.ac.jp}
\end{center}

\bigskip

\tableofcontents

\section{Introduction}

\subsection{Background and Motivation}
\label{subsec:background-motivation}

The classical theory of complex dynamics begins with the iteration of a single
holomorphic map, as developed in the pioneering works of Julia and Fatou.
A central feature of this theory is the decomposition into stable and unstable
regions, represented by the Fatou and Julia sets
\cite{JMPA_1918_8_1__47_0,BSMF_1919__47__161_0}; see, for instance,
\cite{MR1128089,MR1230383,MR2193309} for standard references.
A natural generalization is to study compositions of several maps rather than
iterates of one fixed map.
This viewpoint led to the theory of rational semigroups initiated by
Hinkkanen and Martin \cite{MR1397693,MR1429333}.

\smallskip
Randomization adds another layer to this picture.
At each step, the map is chosen according to a probability law, and one studies the resulting random dynamics.
Random complex dynamics was initiated by Forn{\ae}ss and Sibony \cite{MR1145616}
and further developed by Sumi and others; see, for instance,
\cite{MR2323605,MR2747724,MR3084426,MR3333714,MR4002398,MR4268827,MR4407234,MR4803667}.
One of the central themes in this direction is the study of the probability
that a random orbit tends to infinity, that is, the escaping probability.
In the polynomial setting, this quantity is closely related to the structure
of Julia sets, filled-in Julia sets, and averaged dynamical behavior.

\smallskip
In the i.i.d.\ setting, Sumi \cite{MR2747724} developed a detailed theory of random complex dynamics
for independently chosen rational maps and established a family of cooperation principles.
In particular, Cooperation Principle~I asserts that if the kernel Julia set is empty,
then the chaos of the averaged dynamics disappears.
This is a genuinely randomness-induced phenomenon, which has no deterministic analogue.
Later, Sumi and Watanabe \cite{MR4002398} extended this approach to non-i.i.d.\ random dynamics
driven by a Markov chain,
using the framework of graph-directed Markov systems.
Subsequently, in \cite{endo2026juliafatoutheoryrandomsystems}, the author
developed a Julia--Fatou theory for random dynamical systems with complete
connections generated by continuous maps on compact metric spaces and
further generalized the cooperation principle under suitable assumptions.
In \cite{MR2747724,MR4002398}, the cooperation principle was used to analyze
the probability of tending to infinity. Following this line of work, we study
the corresponding probability in the RSCC setting by applying the cooperation principle from
\cite{endo2026juliafatoutheoryrandomsystems}.

\smallskip
We use the framework of random systems with complete connections (RSCCs) in
this paper. It is rooted in the notion of dependence with complete
connections introduced by Onicescu and Mihoc and developed systematically by
Iosifescu and Grigorescu \cite{MR1070097}.

\smallskip
RSCCs provide a flexible probabilistic framework for stochastic processes that are not necessarily Markovian and may exhibit dependence on the entire past.
Because of this generality, they appear in a variety of mathematical contexts.
In fractal geometry and ergodic theory, Mihailescu and Urba\'nski \cite{MR4396787} develop a unified framework encompassing countable iterated function systems with overlaps, Smale endomorphisms, and RSCCs, and establish geometric and dimensional properties of stationary measures.
Connections with statistical mechanics have also been explored in \cite{MR2041831,MR2123648}, where RSCCs are reformulated within a framework parallel to Gibbs theory, thereby clarifying their thermodynamic features.
In number theory, RSCCs have been applied to continued fraction expansions;
see \cite{MR3035955,MR3523394,MR4075402}.
In time series analysis, Truquet \cite{MR4140544} provides general conditions ensuring stationarity, ergodicity, and mixing properties for infinite-memory processes, including observation-driven models with exogenous covariates arising in finance, economics, and climate data.
However, apart from \cite{endo2026juliafatoutheoryrandomsystems}, to the best
of the author's knowledge, RSCCs have not been systematically used as a
foundation for random complex dynamics, in particular for the study of the
probability that a random orbit tends to infinity.
It is also natural to consider a more general setting beyond RSCCs, such as
a skew product over a prescribed base transformation, but we do not discuss
it further in this paper. We restrict ourselves to the framework of
\cite{endo2026juliafatoutheoryrandomsystems}, in which the cooperation
principle has already been established.

\smallskip

\smallskip
Our main object in this paper is the escaping probability
\[
T_{\infty,\tilde\tau_w}(z),
\]
which describes the probability that the random orbit starting from
$z\in\widehat{\mathbb C}$ tends to $\infty$ when the initial state is $w\in W$.
We study this quantity from two complementary viewpoints.
The first is the \emph{statewise} viewpoint, in which we fix the initial state and analyze
the dependence on the spatial variable.
The second is the \emph{stationary-averaged} viewpoint, in which we average the statewise escaping probabilities
with respect to stationary distributions of the induced state chain.
This second viewpoint is particularly natural in the RSCC setting, since the state variable itself evolves dynamically and stationary distributions encode the long-term statistical behavior of the selection mechanism.

\smallskip
The aim of this paper is to clarify how the RSCC structure influences the geometry of non-escaping sets,
the regularity of escaping probabilities, and the effect of stationary averaging.
In particular, the examples in the final section show that reinforcement-type state dependence may create phenomena
that do not appear in the i.i.d.\ case, while suitable truncations can restore continuity.
In this sense, the paper may be viewed as a first step toward a theory of
\emph{RSCC-induced phenomena} in random complex dynamics.

\subsection{Setting and Main Results}
\label{subsec:main-results}

We briefly describe the setting and the main results. Precise definitions and
standing assumptions are given in the subsequent sections.

We consider random polynomial dynamics driven by an RSCC. Each index
$x\in X$ is assigned a Borel probability measure $\tau_x$ on $\Poly$, with
support $\Gamma_x$, and we denote the resulting polynomial RSCC by $S_\tau$.
For each initial state $w$, these data determine a probability measure
$\tilde\tau_w$ on the admissible polynomial paths.

For $z\in\widehat{\mathbb C}$, we denote by
$T_{\infty,\tilde\tau_w}(z)$ the probability that the corresponding
random orbit starting from $z$ tends to infinity. We denote the Fatou and
kernel Julia sets at $w$ by $F_w(S_\tau)$ and $J_{\ker,w}(S_\tau)$,
respectively, and write
\[
\mathbb T_{\infty,\tau}(z,w)
:=
T_{\infty,\tilde\tau_w}(z)
\]
for the escaping probability on the product space
$\widehat{\mathbb C}\times W$. Precise definitions are given in
Sections~\ref{sec:preliminaries} and~\ref{sec:statewise-escaping}.

Our first result concerns the regularity of these escaping probabilities.

\begin{mainresult}[Lemma~\ref{lem:escape-local-constancy-Fatou}, Theorem~\ref{thm:continuity-product-escape}, and Corollary~\ref{cor:continuity-statewise-escape}]
\label{mr:statewise-regularity}
For each state $w\in W$, the statewise escaping probability
$T_{\infty,\tilde\tau_w}\colon \widehat{\mathbb C}\to[0,1]$
is locally constant on the Fatou set $F_w(S_\tau)$.
If
$J_{\ker,w}(S_\tau)=\emptyset$ for all $w\in W$, then the product-space
escaping function
$\mathbb T_{\infty,\tau}(z,w):=T_{\infty,\tilde\tau_w}(z)$
is continuous on $\widehat{\mathbb C}\times W$. In particular, for each fixed
state $w\in W$, the function $T_{\infty,\tilde\tau_w}$ is continuous on
$\widehat{\mathbb C}$.
\end{mainresult}

Thus, the statewise escaping probability is locally constant outside the Julia set.
The simultaneous emptiness of the kernel Julia sets at all states ensures
continuity with respect to both the spatial
and state variables.

For each state $w\in W$, we also consider the set
\[
K_w(S_\tau)
:=
\left\{
z\in\widehat{\mathbb C}
:\,
\{\gamma(z):\gamma\in H_w(S_\tau)\}
\text{ is bounded in } \mathbb C
\right\},
\]
that is, the set of points whose admissible forward images remain bounded in
$\mathbb C$. This set will be called the statewise filled-in Julia set at $w$;
see Definition~\ref{def:statewise-filledin-RSCC} for the precise definition.

The next result identifies the zero set of the escaping
probability with $K_w(S_\tau)$, giving a geometric characterization
of the points from which escape has
probability zero.

\begin{mainresult}[Proposition~\ref{prop:char-filledin-escape-compact}]
\label{mr:filled-in-relation}
For every state $w\in W$,
the zero set of the statewise escaping
probability agrees with the statewise filled-in Julia set:
\[
K_w(S_\tau)
=
\left\{
z\in\widehat{\mathbb C}
:\,
T_{\infty,\tilde\tau_w}(z)=0
\right\}.
\]
Hence the filled-in set is characterized exactly as the set of points from
which escape has probability zero.
\end{mainresult}

We then pass from statewise behavior to stationary-averaged behavior.
We denote by $\Stat$ the set of all stationary distributions of the induced state
chain on $W$; see Definition~\ref{def:stationary-distribution-RSCC}.
For each $\pi\in\Stat$, we define
\[
T_{\infty,\tau,\pi}(z)
:=
\int_W T_{\infty,\tilde\tau_w}(z)\,\pi(dw),
\qquad z\in\widehat{\mathbb C}.
\]
The following result shows that stationary averaging preserves continuity.

\begin{mainresult}[Theorem~\ref{thm:continuity-stationary-averaged}]
\label{mr:stationary-averaging}
Assume that $J_{\ker,w}(S_\tau)=\emptyset$ for all $w\in W$. Then
$T_{\infty,\tau,\pi}$ is continuous on $\widehat{\mathbb C}$ for every
$\pi\in\Stat$.
\end{mainresult}

Our final main result gives a sufficient condition for the stationary-averaged
escaping probability to be non-degenerate.
It is formulated in terms of $\pi$-essential filled-in sets associated with
measurable subsets of the state space.

\begin{mainresult}[Proposition~\ref{prop:recovery-pi-essential}]
\label{mr:positive-escape}
Let $\pi\in\Stat$.
If there exist measurable sets $B_1,B_2\subset W$ with positive
$\pi$-measure such that the corresponding $\pi$-essential filled-in sets are
both nonempty and disjoint, then $T_{\infty,\tau,\pi}(z)>0$ for every
$z\in\widehat{\mathbb C}$, and there exists at least one point
$z_0\in\widehat{\mathbb C}$ for which $T_{\infty,\tau,\pi}(z_0)<1$.
\end{mainresult}

This criterion shows that if two positive-measure subsets of the state space
have nonempty, disjoint essential filled-in sets, then
the stationary-averaged escaping probability is
strictly positive everywhere and is not identically equal to $1$.

\subsection{Organization of the Paper}
\label{subsec:organization}

The paper is organized as follows.

Section~\ref{sec:preliminaries} introduces the RSCC framework and recalls the basic objects needed throughout the paper.
We review the definition of a random system with complete connections, the associated state process, and the polynomial RSCC generated by a family of probability measures on $\Poly$.
We also introduce the statewise Julia set, Fatou set, and kernel Julia set, as well as the product-space transition operator and the form of Cooperation Principle~I used in this paper.

Section~\ref{sec:statewise-escaping} first studies the escaping probability from the statewise point of view.
For each initial state $w\in W$, we define the escaping probability
$T_{\infty,\tilde\tau_w}$ and prove that it is locally constant on the Fatou set.
We then relate its zero set to the statewise filled-in Julia set under a compactness assumption on the family of admissible polynomials.
Using the product-space transition operator together with Cooperation Principle~I, we prove continuity of the product-space escaping function
\(
\mathbb T_{\infty,\tau}(z,w)=T_{\infty,\tilde\tau_w}(z)
\)
under the emptiness of the kernel Julia sets.

Subsection~\ref{sec:stationary-averaged-escaping} turns to the stationary-averaged point of view.
We introduce stationary distributions of the induced state chain and define the stationary-averaged escaping probability
\(
T_{\infty,\tau,\pi}
\).
Under the same kernel-emptiness assumption, we prove continuity of
$T_{\infty,\tau,\pi}$.
We also %
give a sufficient condition ensuring that the stationary-averaged escaping
probability is everywhere positive and not identically equal to $1$.

Section~\ref{sec:examples} presents examples illustrating phenomena specific to the RSCC framework.
The first example shows that reinforcement-type state dependence may produce discontinuity of the statewise escaping probability.
The second example shows that this phenomenon disappears after a suitable truncation of the state dynamics.
The third example illustrates how stationary averaging can produce genuinely mixed escaping behavior and provides an application of the sufficient condition established in Subsection~\ref{sec:stationary-averaged-escaping}.

\medskip
To avoid ambiguity, we fix some notation used throughout the paper.
We write
$\mathbb N=\{1,2,3,\dots\}$,
$\mathbb N_0=\{0,1,2,3,\dots\}$,
and use the standard symbols $\mathbb Z$, $\mathbb R$, and $\mathbb C$.
We denote by
$
\widehat{\mathbb C}:=\mathbb C\cup\{\infty\}
$
the Riemann sphere.
For a set $A$, we write $\mathcal P(A)$ for its power set, and for a topological space $A$, we write $\mathcal B(A)$ for its Borel $\sigma$-algebra.
Whenever possible, we follow the notation and terminology of
\cite{MR2747724,MR4002398} in order to facilitate comparison with the existing literature.

\section{Preliminaries}
\label{sec:preliminaries}
In this section, we summarize the cooperation principle for RSCCs.
Much of the material below specializes the framework developed in
\cite{endo2026juliafatoutheoryrandomsystems} for a compact metric space $Y$
to the case $Y=\widehat{\mathbb C}$.
 \begin{definition}[{\cite[Definition~1.1.1]{MR1070097}}]\label{def:RSCC}
     A random system with complete connections (RSCC) is a quadruple $\{ (W, \mathcal{W}), (X,\mathcal{X}),u ,P \} $, where
     \begin{enumerate}[label=\textup{(\roman*)}]
         \item $(W, \mathcal{W})$ and $(X,\mathcal{X})$ are arbitrary measurable spaces;
         \item $u \colon W \times X \to W $ is a $(\mathcal{W} \otimes \mathcal{X}, \mathcal{W})$-measurable map;
         \item $P$ is a transition probability function from $(W,\mathcal{W})$ to $(X,\mathcal{X})$, that is, for each $w \in W$, the map $A \mapsto P(w,A)$ defines a probability 
measure on $(X,\mathcal{X})$, and for each $A \in \mathcal{X}$, 
the map $w \mapsto P(w,A)$ is $\mathcal{W}$-measurable.
     \end{enumerate}
 \end{definition}
 
In the above definition, we call $W$ the \emph{state space}, $X$ the \emph{index space}, and $u$ the \emph{update map}. 

The abstract definition of an RSCC allows arbitrary measurable state and index
spaces and does not impose topological continuity on $u$ or $P$.
The results of the present paper, however, concern the compact continuous
subclass described in Assumption~\ref{ass:standing-cooperation}.
In particular, $X$ is finite and discrete, $W$ is a compact metric space, and
the stated continuity and uniformity conditions are imposed.
For examples illustrating the general RSCC framework, see
\cite[Section~1.2]{MR1070097}.

\begin{notation}
We write $x^{(n)}=(x_{1},\ldots,x_{n})\in X^{n}$.
For $n\in\mathbb{N}$, we define the maps
\[
u^{(n)} \colon W \times X^{n} \to W
\]
inductively by
\[
u^{(n)}(w,x^{(n)})
=
\begin{cases}
u(w,x_1), & \text{if } n=1,\\[6pt]
u\big(u^{(n-1)}(w,x^{(n-1)}),x_n\big),
& \text{if } n\ge2.
\end{cases}
\]
For simplicity, we write $u^{(n)}(w, x^{(n)})$ as $w x^{(n)}$ whenever no confusion arises.
For $r\in\mathbb{N}$, we define the $r$-step transition probability
$P_r$ from $(W,\mathcal W)$ to $(X^r,\mathcal X^r)$ by
\[
P_r(w,A)
=
\int_{X^r}
\mathbf{1}_A(x^{(r)})
\, P(w,dx_1)
P(w x_1,dx_2)
\cdots
P(w x^{(r-1)},dx_r),
\]
for $w\in W$ and $A\in\mathcal X^r$,
with the convention that $P_1(w,A)=P(w,A)$.

For $n,r\in\mathbb{N}$, $w\in W$, and $A\in\mathcal X^r$, we define
\[
P_r^{\,n}(w,A)
:=
P_{n+r-1}\big(w, X^{n-1}\times A\big).
\]
\end{notation}

The following theorem constructs the index process
and the induced state chain.

\begin{theorem}[{\cite[Theorem~1.1.2]{MR1070097}}]\label{thm:RSCC-existence}
Let $\{(W,\mathcal{W}), (X,\mathcal{X}), u, P\}$ be an RSCC, 
and fix an arbitrary state $w_{0}\in W$. 
Then there exists a unique probability measure $\mathbf{P}_{w_{0}}$ on 
$(X^{\mathbb{N}}, \mathcal{X}^{\mathbb{N}})$ 
and a sequence of $X$-valued random variables $(\xi_{n})_{n\in\mathbb{N}}$ defined on 
$(X^{\mathbb{N}}, \mathcal{X}^{\mathbb{N}}, \mathbf{P}_{w_{0}})$ such that, 
for all $m,n,r\in\mathbb{N}$ and $A\in\mathcal{X}^{r}$, the following hold:
\begin{enumerate}[label=\textup{(\roman*)}]
    \item $\mathbf{P}_{w_{0}}\big(
    (\xi_{n},\ldots,\xi_{n+r-1})
    \in A\big)=P_{r}^{n}(w_{0},A)$;
    \item $\mathbf{P}_{w_{0}}\big(
    (\xi_{n+m},\ldots,\xi_{n+m+r-1})
    \in A \mid \xi^{(n)}\big)
    =P_{r}^{m}(w_{0}\xi^{(n)},A)$, $\mathbf{P}_{w_{0}}$-a.s.;

    \item $\mathbf{P}_{w_{0}}\big(
    (\xi_{n+m},\ldots,\xi_{n+m+r-1})
    \in A \mid \xi^{(n)},\zeta^{(n)}\big)
    =P_{r}^{m}(\zeta_{n},A)$, $\mathbf{P}_{w_{0}}$-a.s.;
\end{enumerate}
where $\xi^{(n)}=(\xi_{1},\ldots,\xi_{n})$, $\zeta_{n}=w_{0}\xi^{(n)}$, and $\zeta^{(n)}=(\zeta_{1},\ldots,\zeta_{n})$.

Moreover, the sequence
$(\zeta_{n})_{n\in\mathbb{N}_{0}}$
with $\zeta_{0}=w_{0}$
forms a $W$-valued homogeneous Markov chain whose transition operator
\begin{equation*}
     \U f(w) = \int_{X} f(wx)\,P(w,dx) , 
    \qquad f\in B_b(W,\mathcal W),
\end{equation*}
acts on the Banach space $B_b(W,\mathcal W)$ 
of all bounded $\mathcal{W}$-measurable complex-valued functions on $W$.
\end{theorem}

\begin{notation}
\label{not:Q-Markov-kernel}
Let $\{(W,\mathcal W),(X,\mathcal X),u,P\}$ be an RSCC.
We define the
Markov transition kernel 
$Q$ on $(W,\mathcal W)$ induced by $(u,P)$ by
\[
Q(w,A)
:=
\int_X \mathbf 1_A\bigl(u(w,x)\bigr)\,P(w,dx),
\qquad w\in W,\ A\in\mathcal W.
\]

The associated Markov operator $\U$ on $B_b(W,\mathcal W)$ is given by
\[
(\U f)(w)
:=
\int_W f(w')\,Q(w,dw')
=
\int_X f\bigl(u(w,x)\bigr)\,P(w,dx),
\qquad f\in B_b(W,\mathcal W),
\]
and its iterates satisfy
\[
(\U^{n}f)(w)=\int_W f(w')\,Q^{n}(w,dw'),
\qquad n\in\mathbb N.
\]

Dually, we denote by $ba(W,\mathcal W)$ the space of all finitely additive
complex-valued set functions on $\mathcal W$ with finite total variation, equipped with the total
variation norm. 
We define $\V:ba(W,\mathcal W)\to ba(W,\mathcal W)$ by
\[
(\V\mu)(A)
:=
\int_W Q(w,A)\,\mu(dw),
\qquad A\in\mathcal W.
\]
Then $\V$ is a bounded linear operator with $\|\V\|=1$, and
\[
(\V^{n}\mu)(A)=\int_W Q^{n}(w,A)\,\mu(dw),
\qquad \mu\in ba(W,\mathcal W),\ n\in\mathbb N.
\]
\end{notation}

For each initial state $w\in W$, Theorem~\ref{thm:RSCC-existence} yields an
associated sequence of $X$-valued random variables
$(\xi_n)_{n\in\mathbb N}$ together with a $W$-valued homogeneous Markov chain
$(\zeta_n)_{n\in\mathbb N_0}$.
Using these sequences, we study the induced random complex dynamics.

\begin{notation}
We denote the Riemann sphere by $\widehat{\mathbb{C}}$. 
We define
\[
\Poly
:= \{\, h:\widehat{\mathbb{C}}\to\widehat{\mathbb{C}}
\mid h \text{ is a non-constant polynomial} \,\},
\]
and endow $\Poly$ with the metric
\[
d_{\Poly}(f,g)
:= \sup_{z\in\widehat{\mathbb{C}}}
d_{\widehat{\mathbb{C}}}(f(z),g(z)),
\]
where $d_{\widehat{\mathbb{C}}}$ denotes the spherical metric on
$\widehat{\mathbb{C}}$.  
We further set
\[
\Poly_{+}
:= \{\, h\in\Poly : \deg(h)\ge 2 \,\}.
\]
Both $\Poly$ and $\Poly_{+}$ are equipped with the metric topology induced by
$d_{\Poly}$.
\end{notation}

We denote by $\mathfrak{M}_{1}(\Poly)$ the space of all Borel probability measures
on $\Poly$.
For each $x\in X$, we fix a Borel probability measure
\(
\tau_{x} \in \mathfrak{M}_{1}(\Poly),
\)
and denote by
\[
\Gamma_{x} := \supp \tau_{x} \subset \Poly
\]
the support of $\tau_{x}$.
In what follows, we study the system
\[
S_{\tau}
:=
\{(W,\mathcal{W}), (X,\mathcal{X}), u, P, \{\Gamma_{x}\}_{x\in X}\},
\]
which consists of an RSCC together with a family of subsets
$\{\Gamma_{x}\}_{x\in X}$ of\/ $\Poly$.
We call $S_{\tau}$ a \emph{random system of polynomial maps on
$\widehat{\mathbb{C}}$ with complete connections}, or simply a
\emph{polynomial RSCC}.

For $n\in\mathbb N$ and
$a^{(n)}=(a_1,\ldots,a_n)\in X^n$, we denote by
\[
[a^{(n)}]=[a_1,\ldots,a_n]
:=
\bigl\{
(x_j)_{j\in\mathbb N}\in X^{\mathbb N}:
x_j=a_j \text{ for } j=1,\ldots,n
\bigr\}
\]
the cylinder set determined by $a^{(n)}$.

\begin{assumption}\label{ass:standing-cooperation}
Throughout this paper, we impose the following standing assumptions.
\begin{enumerate}[label=\textup{(\roman*)}, leftmargin=2.8em]

\item\label{ass:finite-index-space}
The index set $X$ is finite and is endowed with the discrete topology,
and $\mathcal X=\mathcal P(X)$. The state space $W$ is a compact metric
space, and $\mathcal W=\mathcal B(W)$.

\item\label{ass:update-equicontinuity}
The family
\[
\bigl\{
W\ni v\longmapsto va^{(n)}\in W:
n\in\mathbb N,\ a^{(n)}\in X^n
\bigr\}
\]
is equicontinuous on $W$.

\item\label{ass:prefix-continuity}
For every $w\in W$, we have
\[
\lim_{v\to w}
\sup_{n\in\mathbb N}
\sum_{a^{(n)}\in X^n}
\left|
\mathbf P_v([a^{(n)}])
-
\mathbf P_w([a^{(n)}])
\right|
=0.
\]

\item\label{ass:compact-polynomial-supports}
For each $x\in X$, the set $\Gamma_x$ is a compact subset of $\Poly_+$.

\end{enumerate}
Throughout this paper, $d_W$ denotes the metric on $W$.
\end{assumption}

Assumption~\ref{ass:standing-cooperation}\ref{ass:update-equicontinuity}
implies the continuity of each one-step update map,
because the family appearing there contains the maps
\[
W\ni v\longmapsto vx=u(v,x)\in W,
\qquad x\in X.
\]
Moreover, since
\[
\mathbf P_v([x])=P(v,\{x\})
\qquad (v\in W,\ x\in X),
\]
Assumption~\ref{ass:standing-cooperation}\ref{ass:prefix-continuity}
implies that, for every $w\in W$,
\[
\lim_{v\to w}
\sum_{x\in X}
\left|
P(v,\{x\})-P(w,\{x\})
\right|
=0.
\]
In particular, the map
\(
W\ni v\longmapsto P(v,\{x\})\in[0,1]
\)
is continuous for every $x\in X$.
Consequently,
\(
\U(C(W))\subset C(W),
\)
so the induced state chain is Feller.

\begin{notation}
We consider a polynomial RSCC
\(
S_{\tau}
=
\{(W,\mathcal{W}), (X,\mathcal{X}), u, P, \{\Gamma_x\}_{x\in X}\}
\)
and fix a
state $w\in W$.

\begin{itemize}
\item
For $n\in\mathbb{N}$, we write $x^{(n)}=(x_1,\ldots,x_n)\in X^n$ and define the set
of \emph{admissible index words of length $n$ from $w$} by
\[
X_{w,n}
:=
\bigl\{
  x^{(n)} \in X^{n}
  \ : \
  \mathbf{P}_{w}\big([x^{(n)}]\big) > 0
\bigr\}.
\]
We also set
\[
X_{w,*}
:=
\bigcup_{n\in\mathbb{N}} X_{w,n},
\]
and define the set of \emph{reachable states from $w$} by
\[
\operatorname{Reach}(w)
:=
\bigl\{
  w x^{(n)} \in W
  \ :\
  x^{(n)} \in X_{w,*}
\bigr\}.
\]

\item
For a finite sequence $\gamma^{(n)} = (\gamma_{1},\ldots,\gamma_{n}) 
\in \Poly^{n}$ and integers $1 \le M \le N \le n$, we define
\(
\gamma_{N,M} := \gamma_{N} \circ \cdots \circ \gamma_{M}.
\)
For an infinite sequence $\gamma = (\gamma_{k})_{k\in\mathbb{N}} 
\in \Poly^{\mathbb{N}}$ and integers $1 \le M \le N$, we define
\(
\gamma_{N,M} := \gamma_{N} \circ \cdots \circ \gamma_{M}.
\)

\item
For each $w\in W$, we define the set of \emph{admissible finite polynomial
compositions from $w$} by
\[
H_w(S_{\tau})
:=
\bigl\{
  \gamma_n\circ\cdots\circ\gamma_1 \in \Poly
  \, :\,
  n\in\mathbb{N},\;
  x^{(n)}=(x_1,\ldots,x_n)\in X_{w,n},\;
  \gamma_j\in\Gamma_{x_j}\ (j=1,\ldots,n)
\bigr\}.
\]
For $v\in\operatorname{Reach}(w)$, we further set
\[
H_w^{v}(S_{\tau})
:=
\bigl\{
  \gamma_n\circ\cdots\circ\gamma_1 \in H_w(S_{\tau})
  \,:\,
  w x^{(n)} = v
\bigr\},
\]
and define $H_w^{v}(S_{\tau}) := \emptyset$ for
$v\notin\operatorname{Reach}(w)$.

\item
Finally, we define the set of all \emph{admissible infinite polynomial sequences
from $w$} by
\[
\Xi_w(S_{\tau})
:=
\Bigl\{
  (\gamma_n,x_n)_{n\in\mathbb{N}} \in (\Poly\times X)^{\mathbb{N}}
  \ :\
  (x_n)_{n\in\mathbb{N}}\in\supp(\mathbf{P}_w),\ 
  \gamma_n\in\Gamma_{x_n}
  \text{ for all } n\in\mathbb{N}
\Bigr\}.
\]
\end{itemize}
\end{notation}

\begin{definition}[Julia set along a path]
\label{def:Julia-path}
For each infinite sequence
\(
\xi=(\gamma_n,x_n)_{n\in\mathbb N}
\in(\Poly\times X)^{\mathbb N},
\)
we define the \emph{Julia set along $\xi$} by
\[
J_\xi
:=
\left\{
z\in\widehat{\mathbb C}:
\{\gamma_{n,1}\}_{n\in\mathbb N}
\text{ is not equicontinuous on any neighborhood of }z
\right\}.
\]
\end{definition}

\begin{definition}
\label{def:Pwtilde}
We consider a polynomial RSCC
\(
S_{\tau}
\)
and fix
$w\in W$.
For $n\in\mathbb{N}$, we fix an admissible word $x^{(n)}=(x_1,\ldots,x_n)\in X_{w,n}$ and take
$A_k\in\mathcal{B}(\Gamma_{x_k})$ for $1\le k\le n$, where
$\mathcal{B}(\Gamma_{x_k})$ denotes the Borel $\sigma$-algebra on
$\Gamma_{x_k}\subset\Poly$.
We define the corresponding cylinder set by
\[
C(A_1,\ldots,A_n; x^{(n)})
:=
\Bigl\{
(\gamma_j,b_j)_{j\in\mathbb N}\in\Xi_w(S_{\tau}) :
b_k=x_k,\ 
\gamma_k\in A_k \ (1\le k\le n)
\Bigr\}.
\]

We denote the collection of all such cylinder sets by $\mathcal{C}_w$.
We define a set function $\tilde{\tau}_{w}$ on $\mathcal{C}_w$ by
\[
\tilde{\tau}_{w}
\bigl(C(A_1,\ldots,A_n; x^{(n)})\bigr)
:=
\mathbf{P}_{w}\bigl([x_1,\ldots,x_n]\bigr)
\prod_{k=1}^{n}\tau_{x_k}(A_k),
\qquad n\in\mathbb N.
\]

We set $\mathcal G_w:=\sigma(\mathcal C_w)$. For $n\in\mathbb N$, we define the
probability kernel
\[
K_{w,n}\bigl(((g_j,b_j))_{j=1}^{n-1},B\bigr)
:=
\sum_{a\in X}P(wb^{(n-1)},\{a\})
\int_{\Poly}\mathbf 1_B(g,a)\,d\tau_a(g),
\]
where $B\in\mathcal B(\Poly\times X)$. For $n=1$, the first argument is
the empty tuple and $wb^{(0)}:=w$.

By the Ionescu--Tulcea extension theorem, these kernels determine a unique
probability measure on $(\Poly\times X)^{\mathbb N}$. Its index marginal is $\mathbf P_w$. By
Assumption~\ref{ass:standing-cooperation}\ref{ass:finite-index-space},
$X$ is finite. Since $\Gamma_a=\supp\tau_a$ for every $a\in X$, the
measure assigns full measure to $\Xi_w(S_\tau)$. Restricting this measure to
$(\Xi_w(S_\tau),\mathcal G_w)$ gives the unique probability measure with
the prescribed cylinder values. We denote this probability measure by
$\tilde\tau_w$.
\end{definition}

We now recall the definitions of the Julia set, the Fatou set, and the kernel
Julia set in the setting of random systems with complete connections.

\begin{definition}
\label{def:Julia-Fatou-kernel-RSCC}
For each state $w\in W$, we define the \emph{Julia set at $w$} by
\[
J_w(S_\tau)
:=
\left\{
z\in \widehat{\mathbb C}
\ : \
H_w(S_\tau) \text{ is not equicontinuous on any neighborhood of } z
\right\}.
\]
We define the corresponding \emph{Fatou set at $w$} by
\(
F_w(S_\tau)
:=
\widehat{\mathbb C}\setminus J_w(S_\tau).
\)
We define the \emph{kernel Julia set at $w$} by
\[
J_{\ker,w}(S_\tau)
:=
\bigcap_{v\in \operatorname{Reach}(w)}
\ \bigcap_{\gamma\in H_w^v(S_\tau)}
\gamma^{-1}\!\bigl(J_v(S_\tau)\bigr).
\]
\end{definition}

\begin{notation}
\label{not:product-space-basic}
We use the compact metric state space $(W,d_W)$ from
Assumption~\ref{ass:standing-cooperation}\ref{ass:finite-index-space}.
We set
\[
\mathbb{Y}:=\widehat{\mathbb C}\times W,
\]
and equip $\mathbb{Y}$ with the metric
\[
d_{\mathbb Y}\bigl((z,w),(z',w')\bigr)
:=
\max\bigl\{
d_{\widehat{\mathbb C}}(z,z'),\,d_W(w,w')
\bigr\}.
\]
We denote by $C(\mathbb{Y})$ the Banach space of all complex-valued continuous
functions $\phi:\mathbb{Y}\to\mathbb{C}$ equipped with the supremum norm
\(
\|\phi\|_{\infty}
:=\sup_{(z,w)\in\mathbb{Y}}|\phi(z,w)|.
\)
Since $\mathbb{Y}$ is a compact metric space, the normed space
$\bigl(C(\mathbb{Y}),\|\cdot\|_{\infty}\bigr)$ is separable.
We denote by $B_b(\mathbb Y)$ the Banach space of all bounded
$\mathcal B(\mathbb Y)$--measurable complex-valued functions on
$\mathbb Y$, equipped with the supremum norm.
We write $\mathfrak{M}_{1}(\mathbb{Y})$ for the set of all Borel probability measures
on $\mathbb{Y}$.
We endow $\mathfrak{M}_{1}(\mathbb{Y})$ with the weak-$*$ topology,
that is, the coarsest topology for which every map
\[
\mathfrak{M}_{1}(\mathbb{Y})\ni \mu
\longmapsto
\int_{\mathbb{Y}} \phi \, d\mu,
\qquad \phi\in C(\mathbb{Y}),
\]
is continuous.
The space $\mathfrak{M}_{1}(\mathbb{Y})$ is compact and metrizable in the
weak-$*$ topology. This topology is induced by the metric
\[
d_{\mathfrak{M}_{1}(\mathbb{Y})}(\mu_1,\mu_2)
:=
\sum_{j=1}^{\infty}
\frac{1}{2^j}\,
\frac{
\bigl|
\int_{\mathbb{Y}} \phi_j\, d\mu_1
-
\int_{\mathbb{Y}} \phi_j\, d\mu_2
\bigr|
}{
1 +
\bigl|
\int_{\mathbb{Y}} \phi_j\, d\mu_1
-
\int_{\mathbb{Y}} \phi_j\, d\mu_2
\bigr|
},
\qquad \mu_1,\mu_2\in \mathfrak{M}_{1}(\mathbb{Y}),
\]
where $(\phi_j)_{j\in\mathbb{N}}$ is any sequence that is dense in $C(\mathbb{Y})$.
\end{notation}

\begin{definition}[Transition operator]
\label{def:transition-operator}
We define the \emph{transition operator}
\(
M_\tau : B_b(\mathbb Y)\longrightarrow B_b(\mathbb Y)
\)
by
\[
M_{\tau}\phi(z,w)
:=
\sum_{x\in X_{w,1}}
P(w,\{x\})
\int_{\Gamma_x}
\phi\bigl(\gamma(z),\, wx\bigr)\, d\tau_x(\gamma),
\]
for $\phi\in B_b(\mathbb Y)$ and $(z,w)\in\mathbb Y$.
\end{definition}

\begin{lemma}
\label{lem:Mtau-is-Markov}
The transition operator $M_\tau$ maps $C(\mathbb Y)$ into itself.
Moreover, the restriction
\(
M_\tau\colon C(\mathbb Y)\longrightarrow C(\mathbb Y)
\)
is a Markov operator.
That is, the following properties hold:
\begin{enumerate}[label=\textup{(\roman*)}]
\item
$M_\tau \mathbf{1}_{\mathbb Y}=\mathbf{1}_{\mathbb Y}$;
\item
$M_\tau\phi\ge 0$ for every $\phi\in C(\mathbb Y)$ with $\phi\ge 0$.
\end{enumerate}
In particular, $\|M_\tau\|=1$ and
\(
M_\tau^{*}\bigl(\mathfrak{M}_{1}(\mathbb{Y})\bigr)
\subset
\mathfrak{M}_{1}(\mathbb{Y}).
\)
\end{lemma}

\begin{definition}\label{def:Fatou-meas}
We use the transition operator $M_\tau:C(\mathbb{Y})\to C(\mathbb{Y})$ from
Definition~\ref{def:transition-operator} and denote by
\(
M_\tau^{*}:\mathfrak{M}_{1}(\mathbb{Y})\to\mathfrak{M}_{1}(\mathbb{Y})
\)
its adjoint. We denote by
\[
\iota:\mathbb{Y}\to\mathfrak{M}_{1}(\mathbb{Y}),
\qquad
\iota(z,w):=\delta_{(z,w)},
\]
the natural embedding that sends each point to the corresponding Dirac measure.
We define the following sets:
\begin{enumerate}[label=\textup{(\roman*)}]
\item
$F_{\mathrm{meas}}(M_\tau^{*})$ is the set of all
$\mu\in\mathfrak{M}_{1}(\mathbb{Y})$ for which there exists a neighborhood
$U$ of $\mu$ in $\mathfrak{M}_{1}(\mathbb{Y})$ such that the family
$\{(M_\tau^{*})^{n}\}_{n\in\mathbb{N}}$ is equicontinuous on $U$.

\item
$F_{\mathrm{pt}}^{0}(M_\tau^{*})$ is the set of all
$(z,w)\in\mathbb{Y}$ at which the family
$\{(M_\tau^{*})^{n}\circ\iota\}_{n\in\mathbb{N}}$
is equicontinuous.
\end{enumerate}
\end{definition}

Next, we recall
two lemmas 
from \cite{endo2026juliafatoutheoryrandomsystems}, which
generalize the corresponding results in \cite{MR2747724,MR4002398}.

\begin{lemma}[Pointwise equicontinuity via test functions]
\label{lem:pt-equicontinuity-Mtau}
Let $M_\tau:C(\mathbb{Y})\to C(\mathbb{Y})$ be a Markov operator and
$M_\tau^{*}:\mathfrak{M}_{1}(\mathbb{Y})\to\mathfrak{M}_{1}(\mathbb{Y})$ its adjoint.
For $(z,w)\in\mathbb{Y}$, the following are equivalent:
\begin{enumerate}[label=\textup{(\roman*)}]
\item
$(z,w)\in F_{\mathrm{pt}}^{0}(M_\tau^{*})$;

\item
For every $\phi\in C(\mathbb{Y})$, the family $\{M_\tau^{n}\phi\}_{n\ge0}$
is equicontinuous at $(z,w)$.
\end{enumerate}
\end{lemma}

\begin{lemma}[Equivalence of measure-theoretic and pointwise equicontinuity]
\label{lem:Fmeas-Fpt-equivalence}
Let $M_\tau:C(\mathbb{Y})\to C(\mathbb{Y})$ be a Markov operator and
$M_\tau^{*}:\mathfrak{M}_{1}(\mathbb{Y})\to\mathfrak{M}_{1}(\mathbb{Y})$
its adjoint.
Then the following are equivalent:
\begin{enumerate}[label=\textup{(\roman*)}]
\item
$F_{\mathrm{meas}}(M_\tau^{*})=\mathfrak{M}_{1}(\mathbb{Y})$;

\item
$F_{\mathrm{pt}}^{0}(M_\tau^{*})=\mathbb{Y}$.
\end{enumerate}
\end{lemma}

The following auxiliary lemma connects the pathwise Julia sets with the
pointwise equicontinuity of the averaged dynamics.

\begin{lemma}[Zero pathwise Julia probability and pointwise equicontinuity]
\label{lem:zero-probability-implies-Fpt}
Let $(z,w)\in\mathbb Y$. Then
\(
A_{z,w}
:=
\left\{
\xi\in\Xi_w(S_\tau):z\in J_\xi
\right\}
\)
belongs to $\mathcal G_w$. If
\(
\tilde\tau_w(A_{z,w})=0,
\)
then
\(
(z,w)\in F_{\mathrm{pt}}^0(M_\tau^*).
\)
\end{lemma}

\begin{proof}[Proof sketch]
The measurability of $A_{z,w}$ follows by expressing its complement as
a countable union of events on which the restrictions of
$\{\gamma_{n,1}\}_{n\in\mathbb N}$ to a closed spherical ball around
$z$ are uniformly equicontinuous.

We assume that $\tilde\tau_w(A_{z,w})=0$. Given
$\phi\in C(\mathbb Y)$ and $\varepsilon>0$, we can choose a
neighborhood $V$ of $z$ and a measurable set
$E\subset\Xi_w(S_\tau)$ with $\tilde\tau_w(E)$ arbitrarily close to
$1$ such that
\[
d_{\widehat{\mathbb C}}
\bigl(\gamma_{n,1}(z'),\gamma_{n,1}(z)\bigr)
\]
is uniformly small for every $z'\in V$, every $n\in\mathbb N$, and
every $\xi\in E$. Assumption~\ref{ass:standing-cooperation}
\ref{ass:update-equicontinuity} controls the corresponding state
variables uniformly in $n$, while
Assumption~\ref{ass:standing-cooperation}
\ref{ass:prefix-continuity} controls the difference between the
finite-word distributions arising from nearby initial states.

Iterating Definition~\ref{def:transition-operator} gives
\[
(M_\tau^n\phi)(z,w)
=
\int_{\Xi_w(S_\tau)}
\phi\bigl(\gamma_{n,1}(z),wx^{(n)}\bigr)
\,d\tilde\tau_w(\xi).
\]
Splitting this integral over $E$ and its complement, and using the
preceding two uniform estimates, shows that
$\{M_\tau^n\phi\}_{n\ge0}$ is equicontinuous at $(z,w)$.
Lemma~\ref{lem:pt-equicontinuity-Mtau} therefore yields
\(
(z,w)\in F_{\mathrm{pt}}^0(M_\tau^*).
\)
\end{proof}

The following theorem is the polynomial specialization of
Cooperation Principle~I established in
\cite{endo2026juliafatoutheoryrandomsystems}. It extends the corresponding
results in \cite{MR2747724,MR4002398}. We include a proof sketch to indicate
how the assumptions enter the argument.
\begin{theorem}[Cooperation Principle~I]
\label{thm:CP-Riemann}
Let $\lambda$ be a finite Borel measure on $\widehat{\mathbb C}$.
Suppose that $S_\tau$ satisfies
Assumption~\ref{ass:standing-cooperation} and that
\(
J_{\ker,w}(S_\tau)=\emptyset
\quad\text{for all } w\in W.
\)
Then the following assertions hold.
\begin{enumerate}
[label=\textup{(\roman*)}]
\item
$F_{\mathrm{meas}}(M_\tau^{*})=\mathfrak{M}_1(\mathbb{Y})$.

\item
For every $w \in W$ and for $\tilde{\tau}_w$-almost every
$\xi=(\gamma_n,x_n)\in \Xi_w(S_\tau)$, we have
\(
\lambda(J_\xi)=0.
\)
\end{enumerate}
\end{theorem}

\begin{proof}[Proof sketch]
We set
\(
\Omega:=(\Poly\times X)^{\mathbb N}
\)
and regard each $\tilde\tau_w$ as a probability measure on $\Omega$
concentrated on $\Xi_w(S_\tau)$.
For $(z,w)\in\mathbb Y$, we define
\[
q(z,w)
:=
\tilde\tau_w\{\xi\in\Omega:z\notin J_\xi\}.
\]
Thus $q(z,w)$ is the probability that the family of compositions
along a random path is equicontinuous on a neighborhood of $z$.

We first show that $q(z,w)>0$ for every $(z,w)\in\mathbb Y$.
By the definition of $J_{\ker,w}(S_\tau)$, there exist an admissible
word $a^{(m)}=(a_1,\ldots,a_m)$ and maps $g_j\in\Gamma_{a_j}$ such that
\(
g_{m,1}(z)
\in F_{wa^{(m)}}(S_\tau).
\)
Since this Fatou set is open, the same inclusion holds when each
$g_j$ is replaced by a map in a sufficiently small neighborhood
of $g_j$. These neighborhoods have positive $\tau_{a_j}$-measure,
and the word $a^{(m)}$ has positive probability. Hence the
corresponding cylinder has positive $\tilde\tau_w$-measure.
For every path in this cylinder, all subsequent compositions are
equicontinuous near the point reached after $m$ steps, by the
definition of $F_{wa^{(m)}}(S_\tau)$. Composing with the first $m$
maps gives equicontinuity near $z$. Thus $q(z,w)>0$.

For each fixed $w$, the openness of
$\widehat{\mathbb C}\setminus J_\xi$ and Fatou's lemma imply that
$q(\,\cdot\,,w)$ is lower semicontinuous.
To compare the probabilities for different initial states, we set
\[
D(v,w)
:=
\sup_{n\in\mathbb N}
\sum_{a^{(n)}\in X^n}
\left|
\mathbf P_v([a^{(n)}])
-
\mathbf P_w([a^{(n)}])
\right|.
\]
The cylinder formula gives
\[
\bigl|\tilde\tau_v(B)-\tilde\tau_w(B)\bigr|
\le \frac12D(v,w)
\]
for every event $B$ determined by finitely many coordinates:
conditional on an index word, the distributions of the polynomial
maps do not depend on the initial state. Approximation by such
events gives the same bound for every $B\in\mathcal B(\Omega)$.
In particular,
\[
|q(z,v)-q(z,w)|\le\frac12D(v,w)
\qquad(z\in\widehat{\mathbb C}).
\]
Assumption~\ref{ass:standing-cooperation}\ref{ass:prefix-continuity}
gives $D(v,w)\to0$ as $v\to w$. Together with the lower
semicontinuity of $q(\,\cdot\,,w)$, this shows that $q$ is lower
semicontinuous on $\mathbb Y$.
Since $\mathbb Y$ is compact and $q$ is strictly positive, there
exists $\delta>0$ such that
\[
q(z,w)\ge\delta
\qquad\text{for every }(z,w)\in\mathbb Y.
\]

We now fix $(z,w)\in\mathbb Y$ and define
\(
A_z:=\{\xi\in\Omega:z\in J_\xi\}.
\)
For a path $\xi=(\gamma_k,x_k)_{k\in\mathbb N}$, we write
\[
z_n:=\gamma_{n,1}(z),
\qquad
w_n:=w(x_1,\ldots,x_n).
\]
We denote by $\mathcal H_n$ the $\sigma$-algebra generated by the
first $n$ pairs $(\gamma_1,x_1),\ldots,(\gamma_n,x_n)$, and we set
\[
r_n:=\tilde\tau_w(A_z\mid\mathcal H_n).
\]
Thus $r_n$ is the conditional probability of $z\in J_\xi$ under
$\tilde\tau_w$, given these first $n$ pairs.

By the construction of the path measures, conditional on
$\mathcal H_n$, the sequence
$(\gamma_{n+k},x_{n+k})_{k\in\mathbb N}$ has distribution
$\tilde\tau_{w_n}$. If its compositions are equicontinuous near
$z_n$, composing them with $\gamma_{n,1}$ and
including the finitely many initial compositions gives
equicontinuity near $z$. Consequently,
\[
r_n
\le 1-q(z_n,w_n)
\le 1-\delta
\qquad\tilde\tau_w\text{-almost surely}.
\]
Since
\[
\sigma\left(\bigcup_{n\in\mathbb N}\mathcal H_n\right)
=\mathcal B(\Omega),
\]
L\'evy's upward theorem gives
\[
r_n\longrightarrow\mathbf 1_{A_z}
\qquad\tilde\tau_w\text{-almost surely},
\]
where $\mathbf 1_{A_z}$ denotes the indicator of $A_z$.
It follows that $\mathbf 1_{A_z}\le1-\delta$ almost surely.
As $\delta>0$, we conclude that
$\tilde\tau_w(A_z)=0$
for every $(z,w)\in\mathbb Y$.
Lemma~\ref{lem:zero-probability-implies-Fpt} now gives
\(
F_{\mathrm{pt}}^0(M_\tau^*)=\mathbb Y.
\)
By Lemma~\ref{lem:Fmeas-Fpt-equivalence}, we obtain
\(
F_{\mathrm{meas}}(M_\tau^*)=\mathfrak M_1(\mathbb Y),
\)
which proves assertion~\textup{(i)}.

Finally, for every $w\in W$, Tonelli's theorem yields
\[
\int_{\Xi_w(S_\tau)}
\lambda(J_\xi)\,d\tilde\tau_w(\xi)
=
\int_{\widehat{\mathbb C}}
\tilde\tau_w(A_z)\,d\lambda(z)
=0.
\]
Since $\lambda(J_\xi)\ge0$, we have $\lambda(J_\xi)=0$ for
$\tilde\tau_w$-almost every $\xi\in\Xi_w(S_\tau)$.
This proves assertion~\textup{(ii)}.
\end{proof}

\section{Escaping Behavior}
\label{sec:statewise-escaping}
In this section, we first fix an initial state and consider the corresponding
random dynamical system obtained from Theorem~\ref{thm:RSCC-existence}.
We study the probability that the random orbit starting at
\(z\in\widehat{\mathbb C}\) tends to \(\infty\), which we call the escaping
probability. We then study its averages with respect to stationary
distributions of the induced state chain.

\subsection{Statewise Escaping Probability}
\label{subsec:statewise-escape}

\begin{definition}\label{def:escape-probability-tilde-tau-w}
We fix a state $w\in W$.
For each $z\in\widehat{\mathbb C}$, we define
\[
T_{\infty,\tilde{\tau}_w}(z)
:=
\tilde{\tau}_w\!\left(
\left\{
\xi=(\gamma_n,x_n)_{n\in\mathbb N}\in \Xi_w(S_\tau)
\ : \
\gamma_{n,1}(z)\to\infty
\text{ in } \widehat{\mathbb C}
\text{ as } n\to\infty
\right\}
\right).
\]

The quantity $T_{\infty,\tilde{\tau}_w}(z)$ is called the
\emph{probability of tending (or escaping) to infinity starting from $z$ at state $w$}.
The function $T_{\infty,\tilde{\tau}_w}\colon \widehat{\mathbb C}\to[0,1]$
is called the \emph{statewise escaping probability function}.
\end{definition}

In analogy with \cite[Lemma~4.21]{MR4002398}, the following lemma establishes
local constancy of the escaping probability on the Fatou set.

\begin{lemma}
\label{lem:escape-local-constancy-Fatou}
Let $w\in W$. 
The function
\(
T_{\infty,\tilde{\tau}_w}\colon \widehat{\mathbb C}\to[0,1]
\)
is locally constant on the Fatou set $F_w(S_\tau)$.
That is, for every $z\in F_w(S_\tau)$ there exists an open neighborhood
$U\subset \widehat{\mathbb C}$ of $z$ such that
\(
T_{\infty,\tilde{\tau}_w}(y_1)
=
T_{\infty,\tilde{\tau}_w}(y_2)
\)
for all 
\(
y_1,y_2\in U.
\)
\end{lemma}

\begin{proof}
We fix $w\in W$ and $z\in F_w(S_\tau)$.
By definition of the Fatou set, there exists an open neighborhood $U$ of $z$
such that the family $H_w(S_\tau)$ is equicontinuous on $U$
with respect to the spherical metric $d_{\widehat{\mathbb C}}$.
By shrinking $U$ if necessary, we may assume that $U$ is connected.

For a fixed admissible infinite sequence
$\xi=(\gamma_n,x_n)_{n\in\mathbb N}\in\Xi_w(S_\tau)$,
we set $g_n:=\gamma_{n,1}\in H_w(S_\tau)$ and define
\[
E_\xi
:=
\bigl\{
y\in U : g_n(y)\to\infty
\text{ in } \widehat{\mathbb C}
\text{ as } n\to\infty
\bigr\}.
\]

We take $y\in E_\xi$.
Since $X$ is finite and each $\Gamma_x$ is a compact subset of
$\Poly_+$, the set
$\Gamma:=\bigcup_{x\in X}\Gamma_x$ is a compact subset of $\Poly_+$.
Hence there exists $R>1$ such that
\[
|h(\zeta)|\ge 2|\zeta|
\qquad
\text{for every }h\in\Gamma
\text{ and every }\zeta\in\mathbb C\text{ with }|\zeta|\ge R.
\]
We set
\(
V_\infty:=\{\zeta\in\mathbb C:|\zeta|>R\}\cup\{\infty\}.
\)
Every sequence of maps from $\Gamma$ sends each point of
$V_\infty$ to infinity under successive composition.
Since $g_n(y)\to\infty$, we can choose $N\in\mathbb N$ such that
$g_N(y)\in V_\infty$.
By continuity of $g_N$, there exists an open neighborhood
$V\subset U$ of $y$ such that $g_N(V)\subset V_\infty$.
Consequently, $g_n(y')\to\infty$ for every $y'\in V$.
Thus $E_\xi$ is open in $U$.

If $y\in U\setminus E_\xi$, then $g_n(y)\not\to\infty$.
There exist $\eta>0$ and a subsequence $n_k\to\infty$ such that
\[
d_{\widehat{\mathbb C}}(g_{n_k}(y),\infty)\ge \eta
\quad\text{for all }k.
\]
Applying equicontinuity with $\varepsilon=\eta/2$
yields $\delta>0$ such that whenever
$d_{\widehat{\mathbb C}}(y',y)<\delta$,
\[
d_{\widehat{\mathbb C}}(g_{n_k}(y'),g_{n_k}(y))<\eta/2
\quad\text{for all }k.
\]
Consequently,
\[
d_{\widehat{\mathbb C}}(g_{n_k}(y'),\infty)
\ge
d_{\widehat{\mathbb C}}(g_{n_k}(y),\infty)
-
d_{\widehat{\mathbb C}}(g_{n_k}(y'),g_{n_k}(y))
\ge
\eta/2,
\]
and hence $g_n(y')$ does not converge to $\infty$.
Thus $U\setminus E_\xi$ is open and $E_\xi$ is closed in $U$.

Since $U$ is connected, it follows that either $E_\xi=\emptyset$
or $E_\xi=U$.
Therefore, for the fixed sequence $\xi$,
the property $g_n(y)\to\infty$ is independent of $y\in U$.

For $y\in U$, we define
\[
A_y
:=
\bigl\{
\xi\in\Xi_w(S_\tau) :
\gamma_{n,1}(y)\to\infty
\text{ as } n\to\infty
\bigr\}.
\]
The preceding argument implies that $A_{y_1}=A_{y_2}$ for all $y_1,y_2\in U$.
Taking the $\tilde{\tau}_w$-probability of these events gives
\[
T_{\infty,\tilde{\tau}_w}(y_1)
=
T_{\infty,\tilde{\tau}_w}(y_2)
\quad
\text{for all } y_1,y_2\in U.
\]

Hence $T_{\infty,\tilde{\tau}_w}$ is locally constant on $F_w(S_\tau)$.
\end{proof}

\subsection{The Statewise Filled-in Julia Set}
\label{subsec:filledin}

Next, we introduce the
statewise filled-in Julia set and study its properties.

\begin{definition}
\label{def:statewise-filledin-RSCC}
For each $w\in W$, we define the
\emph{
statewise filled-in Julia set of $S_\tau$ at $w$} by
\(
K_w(S_\tau)
:=
\left\{
z\in\widehat{\mathbb C}
\ : \
\{\gamma(z): \gamma\in H_w(S_\tau)\}
\text{ is bounded in } \mathbb C
\right\}.
\)
\end{definition}

\begin{proposition}
\label{prop:char-filledin-escape-compact}
Let $S_\tau$ be a polynomial RSCC.
For each $w\in W$,
the statewise filled-in Julia set
$K_w(S_\tau)$ defined in Definition~\ref{def:statewise-filledin-RSCC} satisfies
\[
K_w(S_\tau)
=
\left\{
z\in\widehat{\mathbb C}
\ : \
T_{\infty,\tilde{\tau}_w}(z)=0
\right\}.
\]
Equivalently,
\[
K_w(S_\tau)
=
\left\{
z\in\widehat{\mathbb{C}}
\ : \
\text{for every }
\xi=(\gamma_n,x_n)_{n\in\mathbb N}\in\Xi_w(S_\tau),\ 
\gamma_{n,1}(z)\not\to\infty
\text{ in }\widehat{\mathbb{C}}
\text{ as } n\to\infty
\right\}.
\]
\end{proposition}

\begin{proof}
We fix $w\in W$ and set
\[
\Gamma:=\bigcup_{x\in X}\Gamma_x\subset\Poly_+ .
\]
By Assumption~\ref{ass:standing-cooperation}\ref{ass:finite-index-space},
$X$ is finite, and by
Assumption~\ref{ass:standing-cooperation}\ref{ass:compact-polynomial-supports},
each $\Gamma_x$ is compact. Therefore $\Gamma$ is compact.

\medskip
\noindent
\textbf{Claim 1.}
There exists an open neighborhood $U$ of $\infty$ in $\widehat{\mathbb C}$ such
that
\(
h(U)\subset U
\)
for all 
\(
h\in\Gamma .
\)
Moreover, if $\gamma_{N,1}(z)\in U$ for some $N\in\mathbb N$, then
$\gamma_{n,1}(z)\to\infty$ in $\widehat{\mathbb C}$ as $n\to\infty$.

\medskip
\noindent
\emph{Proof of Claim 1.}
Since $\Gamma$ is a compact family of polynomials of degree at least two,
there exists $R>1$ such that
\[
|h(z)|\ge 2|z|
\qquad
\text{for every }h\in\Gamma
\text{ and every }z\in\mathbb C\text{ with }|z|\ge R.
\]
We set
\(
U:=\{z\in\mathbb C:|z|>R\}\cup\{\infty\}.
\)
Then $h(U)\subset U$ for every $h\in\Gamma$.
If $\gamma_{N,1}(z)\in U$ for some $N\in\mathbb N$, then
\[
|\gamma_{n,1}(z)|
\ge
2^{\,n-N}|\gamma_{N,1}(z)|
\qquad(n\ge N),
\]
unless $\gamma_{N,1}(z)=\infty$, in which case the conclusion is immediate.
Thus $\gamma_{n,1}(z)\to\infty$.

\medskip
\noindent
\textbf{Claim 2.}
If $z\in K_w(S_\tau)$, then $T_{\infty,\tilde\tau_w}(z)=0$.

\medskip
\noindent
\emph{Proof of Claim 2.}
We take $z\in K_w(S_\tau)$.
By definition of $K_w(S_\tau)$, the set
$\{\gamma(z):\gamma\in H_w(S_\tau)\}$ is bounded in $\mathbb C$, and hence its
closure in $\widehat{\mathbb C}$ does not contain $\infty$.
For any $\xi=(\gamma_n,x_n)_{n\in\mathbb N}\in\Xi_w(S_\tau)$, each finite
composition $\gamma_{n,1}$ belongs to $H_w(S_\tau)$, so the sequence
$(\gamma_{n,1}(z))_{n\in\mathbb N}$ cannot converge to $\infty$.
Therefore the escape event has $\tilde\tau_w$-measure zero, and
$T_{\infty,\tilde\tau_w}(z)=0$.

\medskip
\noindent
\textbf{Claim 3.}
If $T_{\infty,\tilde\tau_w}(z)=0$, then $z\in K_w(S_\tau)$.

\medskip
\noindent
\emph{Proof of Claim 3.}
We assume that $z\notin K_w(S_\tau)$.
Then the set $\{\gamma(z):\gamma\in H_w(S_\tau)\}$ is unbounded in $\mathbb C$,
and hence its closure in $\widehat{\mathbb C}$ contains $\infty$.
Thus there exists $\gamma\in H_w(S_\tau)$ such that $\gamma(z)\in U$, where $U$
is the neighborhood obtained in Claim~1.

We write $\gamma=\gamma_n\circ\cdots\circ\gamma_1$ with
$x^{(n)}=(x_1,\ldots,x_n)\in X_{w,n}$ and $\gamma_j\in\Gamma_{x_j}$.
By continuity of composition and evaluation in the $d_{\Poly}$-topology, there
exist open neighborhoods $A_j$ of $\gamma_j$ in $\Gamma_{x_j}$ such that
\[
(\tilde\gamma_n\circ\cdots\circ\tilde\gamma_1)(z)\in U
\quad\text{for all } \tilde\gamma_j\in A_j .
\]
Since $\gamma_j\in\supp(\tau_{x_j})$, we have $\tau_{x_j}(A_j)>0$, and since
$x^{(n)}\in X_{w,n}$ we have $\mathbf P_w([x_1,\ldots,x_n])>0$.
We set
\(
C:=C(A_1,\ldots,A_n;x^{(n)}).
\)
Then $C\subset\Xi_w(S_\tau)$ and $\tilde\tau_w(C)>0$.

For any $\xi\in C$, we have $\gamma_{n,1}(z)\in U$, and hence
$\gamma_{m,1}(z)\to\infty$ in $\widehat{\mathbb C}$ as $m\to\infty$ by Claim~1.
Therefore $T_{\infty,\tilde\tau_w}(z)\ge\tilde\tau_w(C)>0$, a contradiction.
This proves $z\in K_w(S_\tau)$.

\medskip
Claims~1--3 together imply
\(
K_w(S_\tau)
=
\{z\in\widehat{\mathbb C}:T_{\infty,\tilde\tau_w}(z)=0\}.
\)
The equivalence with the pathwise characterization follows immediately from
Claim~1.
\end{proof}

\begin{proposition}
\label{prop:boundary-filledin-escape}
For each $w\in W$, we have
\(
\partial K_w(S_\tau)\subset J_w(S_\tau).
\)
\end{proposition}

\begin{proof}
We fix $w\in W$ and $z\in F_w(S_\tau)$.
By Lemma~\ref{lem:escape-local-constancy-Fatou}, there exists
an open neighborhood $U$ of $z$ on which
$T_{\infty,\tilde\tau_w}$ is constant.
By Proposition~\ref{prop:char-filledin-escape-compact},
\[
K_w(S_\tau)
=
\{y\in\widehat{\mathbb C}:T_{\infty,\tilde\tau_w}(y)=0\}.
\]
Hence either $U\subset K_w(S_\tau)$ or
$U\cap K_w(S_\tau)=\emptyset$.
In either case, $z\notin\partial K_w(S_\tau)$.
Therefore,
$\partial K_w(S_\tau)\subset J_w(S_\tau)$.
\end{proof}

\subsection{A Fixed Point Representation via the Transition Operator}
\label{subsec:operator-representation}
In this subsection, we use the statewise escaping probabilities studied above
to define
$\mathbb{T}_{\infty,\tau}$
on the product space
\(\mathbb{Y}:=\widehat{\mathbb C}\times W\).
We show that this function is a fixed point of the transition operator and
apply Cooperation Principle~I to prove its continuity under suitable
assumptions.
The results in this subsection may be regarded as an RSCC analogue of \cite[Proposition~4.24]{MR4002398}.

\begin{definition}
\label{def:escape-product}
We define a function
\(
\mathbb{T}_{\infty,\tau}
\colon \mathbb Y \longrightarrow [0,1]
\)
by
\[
\mathbb{T}_{\infty,\tau}(z,w)
:=
T_{\infty,\tilde{\tau}_w}(z),
\qquad (z,w)\in \mathbb Y.
\]
\end{definition}

\begin{proposition}
\label{prop:approx-escape-Mtau}
Let $U\subset\widehat{\mathbb C}$ be an open neighborhood of $\infty$ such that
\begin{equation}\label{eq:common-U-forward}
h(U)\subset U
\quad\text{for all }h\in\Gamma:=\bigcup_{x\in X}\Gamma_x.
\end{equation}
We additionally assume that $U$ has the common escape property established in
the proof of Proposition~\ref{prop:char-filledin-escape-compact}: for every
$w\in W$, every
$\xi=(\gamma_n,x_n)_{n\in\mathbb N}\in\Xi_w(S_\tau)$, and every
$z\in\widehat{\mathbb C}$,
\[
\gamma_{N,1}(z)\in U
\text{ for some }N\in\mathbb N
\quad\Longrightarrow\quad
\gamma_{n,1}(z)\to\infty.
\]
Let $\phi\in C(\mathbb Y)$ satisfy $\|\phi\|_\infty \le 1$, $\phi(\infty,w)=1$ for all $w\in W$, and $\supp(\phi)\subset U\times W$.
Then the following assertions hold.
\begin{enumerate}[label=\textup{(\roman*)}]
\item
The sequence $\{M_\tau^{n}\phi\}_{n\in\mathbb N}$ converges pointwise on $\mathbb Y$ to $\mathbb T_{\infty,\tau}$.

\item
The function $\mathbb T_{\infty,\tau}$ is a fixed point of $M_\tau$, that is,\ $M_\tau \mathbb T_{\infty,\tau} = \mathbb T_{\infty,\tau}$ on $\mathbb Y$.

\item
If the family $\{M_\tau^{n}\phi\}_{n\in\mathbb N}$ is equicontinuous on $\mathbb Y$, then the convergence in {\rm(i)} is uniform on $\mathbb Y$.
In particular, $\mathbb T_{\infty,\tau}\in C(\mathbb Y)$.
\end{enumerate}
\end{proposition}

\begin{proof}
We fix $(z,w)\in\mathbb Y$.

\medskip
\noindent\textbf{Claim 1.}
For every $n\in\mathbb N$, we have
\begin{equation}\label{eq:Mn-integral-claim}
(M_\tau^{n}\phi)(z,w)
=
\int_{\Xi_w(S_\tau)}
\phi\bigl(\gamma_{n,1}(z),\, w_n\bigr)
\, d\tilde\tau_w(\xi),
\end{equation}
where $\xi=(\gamma_k,x_k)_{k\in\mathbb N}$ and
$w_n:=w x^{(n)}$.

\smallskip
\noindent\emph{Proof of Claim 1.}
We argue by induction on $n$.
For $n=1$, the identity follows directly from
Definition~\ref{def:transition-operator}
and the construction of $\tilde\tau_w$.
We assume that \eqref{eq:Mn-integral-claim} holds for some $n$.
Using the definition of $M_\tau$ and the product structure of
$\tilde\tau_w$, together with Fubini's theorem,
we obtain
\[
(M_\tau^{n+1}\phi)(z,w)
=
\int_{\Xi_w(S_\tau)}
\phi\bigl(\gamma_{n+1,1}(z),\, w_{n+1}\bigr)
\, d\tilde\tau_w(\xi),
\]
which completes the induction.

\medskip
\noindent\textbf{Claim 2.}
We have
\(
\lim_{n\to\infty} (M_\tau^{n}\phi)(z,w)
=
\mathbb T_{\infty,\tau}(z,w).
\)

\smallskip
\noindent\emph{Proof of Claim 2.}
We fix $\xi=(\gamma_k,x_k)_{k\in\mathbb N}\in\Xi_w(S_\tau)$ and set
$z_n:=\gamma_{n,1}(z)$ and $w_n:=w x^{(n)}$.

We first suppose that $z_n\to\infty$ in $\widehat{\mathbb C}$.
By Assumption~\ref{ass:standing-cooperation}\ref{ass:finite-index-space},
the state space $W$ is compact.
Since $\phi\in C(\mathbb Y)$ and $\phi(\infty,w)=1$ for all $w\in W$,
for every $\varepsilon>0$ there exists a neighborhood $V_\infty$ of $\infty$
in $\widehat{\mathbb C}$ such that
\[
|\phi(z',w')-1|<\varepsilon
\quad\text{for all } (z',w')\in V_\infty\times W.
\]
Hence $z_n\in V_\infty$ for all sufficiently large $n$, and therefore
$\phi(z_n,w_n)\to 1$.

We next suppose that $z_n\not\to\infty$.
If $z_N\in U$ for some $N\in\mathbb N$, then the forward invariance
\eqref{eq:common-U-forward} gives $z_n\in U$ for every $n\ge N$, and the common
escape property of $U$ gives $z_n\to\infty$. This contradicts the present
assumption. Hence $z_n\notin U$ for every $n\in\mathbb N$.
Since $\supp(\phi)\subset U\times W$, it follows that
$\phi(z_n,w_n)=0$ for every $n\in\mathbb N$.

Consequently, for each fixed $\xi\in\Xi_w(S_\tau)$ we have
\[
\lim_{n\to\infty}\phi\bigl(\gamma_{n,1}(z),w_n\bigr)
=
\mathbf 1_{\{\xi\in\Xi_w(S_\tau):\,\gamma_{n,1}(z)\to\infty\}}(\xi).
\]
For every $n$ and every $\xi$, we have
\[
\bigl|\phi\bigl(\gamma_{n,1}(z),w_n\bigr)\bigr|
\le
\|\phi\|_\infty
\le
1
.
\]
Therefore, by the dominated convergence theorem applied to
\eqref{eq:Mn-integral-claim}, we obtain
\[
\lim_{n\to\infty}(M_\tau^{n}\phi)(z,w)
=
\int_{\Xi_w(S_\tau)}
\mathbf 1_{\{\xi:\,\gamma_{n,1}(z)\to\infty\}}(\xi)
\, d\tilde\tau_w(\xi)
=
\tilde\tau_w\!\left(
\{\xi\in\Xi_w(S_\tau):\,\gamma_{n,1}(z)\to\infty\}
\right).
\]
By definition, the last expression equals
$T_{\infty,\tilde\tau_w}(z)=\mathbb T_{\infty,\tau}(z,w)$.

\medskip
\noindent\textbf{Claim 3.}
We have $M_\tau \mathbb T_{\infty,\tau} = \mathbb T_{\infty,\tau}$ on $\mathbb Y$.

\smallskip
\noindent\emph{Proof of Claim 3.}
We fix $(z,w)\in\mathbb Y$.
By conditioning on the first step under $\tilde\tau_w$ and using the Markov property, we obtain
\[
\mathbb T_{\infty,\tau}(z,w)
=
\sum_{x\in X_{w,1}} P(w,\{x\})
\int_{\Gamma_x}
\mathbb T_{\infty,\tau}(\gamma(z),wx)
\, d\tau_x(\gamma).
\]
The right-hand side is precisely $(M_\tau \mathbb T_{\infty,\tau})(z,w)$.

\medskip
\noindent\textbf{Claim 4.}
If the family $\{M_\tau^{n}\phi\}_{n\in\mathbb N}$ is equicontinuous on $\mathbb Y$,
then the convergence in Claim~2 is uniform on $\mathbb Y$.
In particular, $\mathbb T_{\infty,\tau}\in C(\mathbb Y)$.

\smallskip
\noindent\emph{Proof of Claim 4.}
By Assumption~\ref{ass:standing-cooperation}\ref{ass:finite-index-space},
$\mathbb Y$ is compact. We also have
\[
\|M_\tau^{n}\phi\|_\infty
\le
\|\phi\|_\infty
\le 1
\quad\text{for all } n,
\]
so the family $\{M_\tau^{n}\phi\}$ is uniformly bounded.
By the Arzelà–Ascoli theorem, every subsequence admits a uniformly
convergent subsubsequence.
By Claim~2, any pointwise limit must coincide with
$\mathbb T_{\infty,\tau}$.
Therefore the full sequence converges uniformly on $\mathbb Y$,
and the limit is continuous as a uniform limit of continuous functions.

\medskip
Claims~1--4 together yield the assertions of the proposition.
\end{proof}

\begin{theorem}
\label{thm:continuity-product-escape}
Assume that
\(
J_{\ker,w}(S_\tau)=\emptyset
\quad\text{for all } w\in W.
\)
Then the function
\[
\mathbb T_{\infty,\tau}\colon \mathbb Y=\widehat{\mathbb C}\times W \longrightarrow [0,1]
\]
is continuous on $\mathbb Y$.
\end{theorem}

\begin{proof}
We choose $\phi\in C(\mathbb Y)$ as in
Proposition~\ref{prop:approx-escape-Mtau}, with respect to a common escape
neighborhood $U$ satisfying \eqref{eq:common-U-forward} and the escape property
stated there.
We set
\[
f_n := M_\tau^{n}\phi \in C(\mathbb Y),
\qquad n\ge 0.
\]
By Proposition~\ref{prop:approx-escape-Mtau}{\rm(i)}, we have pointwise convergence
\[
f_n(z,w)\longrightarrow \mathbb T_{\infty,\tau}(z,w)
\quad \text{for all }(z,w)\in\mathbb Y.
\]

By Theorem~\ref{thm:CP-Riemann}{\rm(i)}, we have
\(
F_{\mathrm{meas}}(M_\tau^{*})=\mathfrak{M}_1(\mathbb Y).
\)
Hence Lemma~\ref{lem:Fmeas-Fpt-equivalence} yields
\(
F_{\mathrm{pt}}^{0}(M_\tau^{*})=\mathbb Y.
\)
Applying Lemma~\ref{lem:pt-equicontinuity-Mtau} to the test function $\phi$,
we conclude that for every $(z,w)\in\mathbb Y$ the family $\{f_n\}_{n\ge 0}$
is equicontinuous at $(z,w)$.

We now fix $(z,w)\in\mathbb Y$ and take $\varepsilon>0$.
By equicontinuity at $(z,w)$, there exists $\delta>0$ such that
\[
d_{\mathbb Y}\bigl((z,w),(z',w')\bigr)<\delta
\ \Longrightarrow\
|f_n(z',w')-f_n(z,w)|<\varepsilon
\quad\text{for all }n\ge 0.
\]
Letting $n\to\infty$ and using the pointwise convergence $f_n\to \mathbb
T_{\infty,\tau}$, we obtain
\[
d_{\mathbb Y}\bigl((z,w),(z',w')\bigr)<\delta
\ \Longrightarrow\
\bigl|\mathbb T_{\infty,\tau}(z',w')-\mathbb T_{\infty,\tau}(z,w)\bigr|
\le \varepsilon.
\]
Thus $\mathbb T_{\infty,\tau}$ is continuous on $\mathbb Y$, that is,
\(\mathbb T_{\infty,\tau}\in C(\mathbb Y)\).
\end{proof}

As an immediate consequence of Theorem~\ref{thm:continuity-product-escape},
we obtain the following corollary.

\begin{corollary}
\label{cor:continuity-statewise-escape}
Assume that
\(
J_{\ker,w}(S_\tau)=\emptyset
\quad\text{for all } w\in W.
\)
Fix $w_0\in W$.
Then the statewise escaping probability function
\(
T_{\infty,\tilde\tau_{w_0}}\colon \widehat{\mathbb C}\to[0,1]
\)
is continuous on $\widehat{\mathbb C}$.
\end{corollary}

\subsection{Stationary-Averaged Escaping Behavior}
\label{sec:stationary-averaged-escaping}
In this subsection,
we study the escaping probability from the viewpoint of
stationary averaging with respect to stationary distributions of the induced
state chain. We introduce the stationary-averaged escaping probability,
prove its continuity under suitable assumptions, and
give a sufficient condition for its non-degeneracy.

\begin{definition}\label{def:stationary-distribution-RSCC}
As in Notation~\ref{not:Q-Markov-kernel}, let $Q$ be the transition probability function on $(W,\mathcal W)$ associated
with the RSCC, and let
\[
\V\mu(\cdot)
:=
\int_W \mu(dw)\, Q(w,\cdot)
\]
be the dual operator acting on $ba(W,\mathcal W)$.
A probability measure $\pi\in\mathfrak M_1(W)$ is called a
\emph{stationary distribution} for the RSCC if
$\V\pi=\pi$,
that is,
\[
\pi(A)
=
\int_W Q(w,A)\,\pi(dw)
\quad
\text{for all } A\in\mathcal W.
\]
We denote by $\Stat$ the set of all stationary distributions, that is,
\[
\Stat
:=
\bigl\{
\pi\in\mathfrak M_1(W)
\ :\ 
\V\pi=\pi
\bigr\}.
\]
\end{definition}

\begin{remark}
If the state space $W$ is finite, then a stationary distribution always exists.
Moreover, if the associated Markov chain is irreducible, the stationary
distribution is unique.
In contrast, for a general Markov chain on an infinite state space, a stationary
distribution need not exist, and even when it exists, it may fail to be unique.

Under Assumption~\ref{ass:standing-cooperation}\ref{ass:finite-index-space},
\ref{ass:update-equicontinuity}, and \ref{ass:prefix-continuity}, however,
the induced state chain is Feller on the compact state space $W$ and
therefore admits at least one stationary distribution.
\end{remark}

\begin{definition}
[{\cite[Definition~3.2.5]{MR1070097}}]
A set $E\in\mathcal W$ is said to be \emph{stochastically closed} if
$Q(w,E)=1$ for all $w\in E$.
\end{definition}

\begin{definition}[Stationary-averaged escaping probability]
\label{def:stationary-averaged-escape}
We fix a stationary distribution $\pi\in\Stat$.
We define the \emph{stationary-averaged escaping probability}
associated with $\pi$ by
\[
T_{\infty,\tau,\pi}(z)
:=
\int_W T_{\infty,\tilde{\tau}_w}(z)\,\pi(dw),
\qquad z\in\widehat{\mathbb C}.
\]
\end{definition}

As an immediate consequence of Theorem~\ref{thm:continuity-product-escape},
we obtain continuity of the stationary-averaged escaping probability.

\begin{theorem}
\label{thm:continuity-stationary-averaged}
Assume that
\(
J_{\ker,w}(S_\tau)=\emptyset
\quad\text{for all } w\in W.
\)
Let $\pi\in\Stat$ be a stationary distribution.
Then the stationary-averaged escaping probability
$T_{\infty,\tau,\pi}$
is continuous on $\widehat{\mathbb C}$.
\end{theorem}

\begin{proof}
By Theorem~\ref{thm:continuity-product-escape}, we have
\(
\mathbb T_{\infty,\tau}\in C(\widehat{\mathbb C}\times W).
\)
By Assumption~\ref{ass:standing-cooperation}\ref{ass:finite-index-space},
$W$ is compact, and hence $\widehat{\mathbb C}\times W$ is compact.
Therefore $\mathbb T_{\infty,\tau}$ is uniformly continuous and bounded.

We take a sequence $(z_n)$ in $\widehat{\mathbb C}$ with
$z_n\to z$.
Then for every $w\in W$,
\[
\mathbb T_{\infty,\tau}(z_n,w)
\longrightarrow
\mathbb T_{\infty,\tau}(z,w),
\]
and
\[
\bigl|
\mathbb T_{\infty,\tau}(z_n,w)
\bigr|
\le 1
\quad
\text{for all } n \text{ and all } w.
\]
Hence, by the dominated convergence theorem,
\[
T_{\infty,\tau,\pi}(z_n)
=
\int_W \mathbb T_{\infty,\tau}(z_n,w)\,\pi(dw)
\longrightarrow
\int_W \mathbb T_{\infty,\tau}(z,w)\,\pi(dw)
=
T_{\infty,\tau,\pi}(z).
\]
Therefore $T_{\infty,\tau,\pi}$ is continuous on
$\widehat{\mathbb C}$.
\end{proof}

The following proposition may be viewed as a state-dependent counterpart of
\cite[Proposition~4.23]{MR4002398}
within the framework of polynomial RSCCs.
It provides a measure-theoretic sufficient condition for the non-degeneracy of the stationary-averaged escaping probability.

\begin{proposition}
\label{prop:recovery-pi-essential}
Let $\pi\in\Stat$.
For $B\in\mathcal W$, we define
\[
K_{B}^{\pi}(S_\tau)
:=
\left\{
z\in\widehat{\mathbb C}
\ : \
\pi\bigl(\{w\in B: z\in K_w(S_\tau)\}\bigr)=\pi(B)
\right\}.
\]
Suppose that there exist $B_1,B_2\in\mathcal W$ such that:
\begin{enumerate}[label=\textup{(\roman*)}]
\item\label{itm:ess-posmass}
$\pi(B_1)>0$ and $\pi(B_2)>0$;
\item\label{itm:ess-commonK1}
$K_{B_1}^{\pi}(S_\tau)\neq\emptyset$;
\item\label{itm:ess-commonK2}
$K_{B_2}^{\pi}(S_\tau)\neq\emptyset$;
\item\label{itm:ess-disjoint}
$K_{B_1}^{\pi}(S_\tau)\cap K_{B_2}^{\pi}(S_\tau)=\emptyset$.
\end{enumerate}
Then there exists $z_0\in\widehat{\mathbb C}$ such that
$T_{\infty,\tau,\pi}(z_0)<1$, and $T_{\infty,\tau,\pi}(z)>0$ for every
$z\in\widehat{\mathbb C}$.
\end{proposition}

\begin{proof}
We choose $z_0\in K_{B_1}^{\pi}(S_\tau)$. By the definition of $K_{B_1}^{\pi}(S_\tau)$, we have
\[
\pi\bigl(\{w\in B_1: z_0\in K_w(S_\tau)\}\bigr)=\pi(B_1),
\]
and hence $z_0\in K_w(S_\tau)$ for $\pi$-almost every $w\in B_1$. By Proposition~\ref{prop:char-filledin-escape-compact}, it follows that
\[
T_{\infty,\tilde{\tau}_w}(z_0)=0
\quad
\text{for $\pi$-almost every } w\in B_1.
\]
Therefore
\[
\int_{B_1} T_{\infty,\tilde{\tau}_w}(z_0)\,\pi(dw)=0,
\]
and consequently
\[
T_{\infty,\tau,\pi}(z_0)
=\int_{W\setminus B_1} T_{\infty,\tilde{\tau}_w}(z_0)\,\pi(dw)
\le \pi(W\setminus B_1)
=1-\pi(B_1)
<1,
\]
where the strict inequality follows from~\ref{itm:ess-posmass}.

We now take an arbitrary $z\in\widehat{\mathbb C}$. By assumption~\ref{itm:ess-disjoint}, there exists $i\in\{1,2\}$ such that $z\notin K_{B_i}^{\pi}(S_\tau)$, which means that
\[
\pi\bigl(\{w\in B_i: z\in K_w(S_\tau)\}\bigr)
<\pi(B_i).
\]
Since $\pi(B_i)>0$ by~\ref{itm:ess-posmass}, it follows that
\[
\pi\bigl(\{w\in B_i: z\notin K_w(S_\tau)\}\bigr)>0.
\]
We set
\(
A:=\{w\in B_i:z\notin K_w(S_\tau)\}.
\)
For every $w\in A$,
Proposition~\ref{prop:char-filledin-escape-compact} yields
\(
T_{\infty,\tilde{\tau}_w}(z)>0.
\)
We define
\(
f(w):=T_{\infty,\tilde{\tau}_w}(z).
\)
By Proposition~\ref{prop:approx-escape-Mtau}{\rm(i)}, the function $f$ is
Borel measurable. For $m\in\mathbb N$, we set
\[
A_m:=\{w\in A:f(w)\ge 1/m\}.
\]
Since $f(w)>0$ for every $w\in A$, we have
\(
A=\bigcup_{m\in\mathbb N}A_m.
\)
Because $\pi(A)>0$, there exists $m\in\mathbb N$ such that
$\pi(A_m)>0$. Therefore
\[
T_{\infty,\tau,\pi}(z)
=
\int_W f(w)\,\pi(dw)
\ge
\int_{A_m}f(w)\,\pi(dw)
\ge
\frac{1}{m}\pi(A_m)
>0.
\]
\end{proof}

\section{Examples}
\label{sec:examples}

In this section, we present three examples of polynomial RSCCs illustrating how state dependence influences both the statewise escaping behavior and the stationary-averaged escaping behavior. The first example exhibits a reinforcement effect that creates discontinuity
of the statewise escaping probability. The second shows how this phenomenon disappears
after truncation of the state dynamics. The 
third provides a model in which the stationary-averaged escaping probability
is strictly positive everywhere but strictly less than \(1\) at some point.

\begin{example}
\label{ex:two-regimes-simplex}
We begin with a reinforcement-type model.

\medskip
\noindent
\textbf{Step 1.} Description of the polynomial RSCC.

We set $X:=\{0,1\}$ and $\mathcal X:=\mathcal P(X)$.
We define two polynomials
\[
f_0(z):=z^2,
\qquad
f_1(z):=\frac{z^2}{2}.
\]
For each $x\in X$, we set
\[
\tau_x:=\delta_{f_x}\in\mathfrak M_1(\Poly),
\qquad
\Gamma_x:=\supp(\tau_x)=\{f_x\}.
\]

We take $W:=[0,1]\subset \mathbb{R}$ with the Euclidean metric and set $\mathcal W:=\mathcal B(W)$.
We fix $\alpha\in[0,1)$ and define the update map
$u:W\times X\to W$ by
\[
u(p,x):=(1-\alpha)p+\alpha x.
\]
We define the transition probabilities by
\[
P(p,\{1\}):=p,
\qquad
P(p,\{0\}):=1-p.
\]

Thus, for each $\alpha\in[0,1)$, we obtain the polynomial RSCC
\[
S_{\tau,\alpha}
=
\{(W,\mathcal W),(X,\mathcal X),u,P,\{\Gamma_x\}_{x\in X}\}
\quad\text{on }\widehat{\mathbb C}.
\]
For every $\alpha\in(0,1)$, the polynomial RSCC
$S_{\tau,\alpha}$ satisfies
Assumption~\ref{ass:standing-cooperation}.
Indeed, for $p,q\in W$ and $a^{(n)}\in X^n$, we have
\[
\left|pa^{(n)}-qa^{(n)}\right|
=
(1-\alpha)^n|p-q|.
\]
Hence the family of iterated update maps is equicontinuous, so
Assumption~\ref{ass:standing-cooperation}\ref{ass:update-equicontinuity}
holds.

To verify
Assumption~\ref{ass:standing-cooperation}\ref{ass:prefix-continuity},
we set
\[
D_n(p,q)
:=
\sum_{a^{(n)}\in X^n}
\left|
\mathbf P_p([a^{(n)}])
-
\mathbf P_q([a^{(n)}])
\right|.
\]
Then
\[
D_{n+1}(p,q)
\le
D_n(p,q)+2(1-\alpha)^n|p-q|.
\]
Consequently,
\[
\sup_{n\in\mathbb N}D_n(p,q)
\le
2|p-q|\sum_{j=0}^{\infty}(1-\alpha)^j
=
\frac{2}{\alpha}|p-q|.
\]
The remaining conditions,
Assumption~\ref{ass:standing-cooperation}\ref{ass:finite-index-space}
and \ref{ass:compact-polynomial-supports},
follow directly from the definitions.

When $\alpha=0$,
Assumption~\ref{ass:standing-cooperation}\ref{ass:prefix-continuity}
fails on the state space $W=[0,1]$.
The case $\alpha=0$ will therefore be treated through singleton
state-space restrictions below.

Here $\alpha\in[0,1)$ quantifies the strength of reinforcement in the state update:
larger $\alpha$ gives stronger dependence of the next state on the current symbol $x$.
In particular, when $\alpha=0$, the state variable remains constant.
For each fixed initial state $p_0\in[0,1]$, the RSCC restricted to the
singleton state space $\{p_0\}$ reduces to an i.i.d.\ random dynamical system
of polynomials as in \cite{MR2747724}.
For $\alpha>0$, the model becomes a genuinely reinforcement-driven,
state-dependent extension of that i.i.d.\ setting.

\medskip
\noindent
\textbf{Step 2.} Statewise Escaping Behavior.

We first consider the boundary states.

If $p=0$, then $P(p,\{0\})=1$, and hence only the map $f_0$ is admissible.
In this case the system reduces to the deterministic dynamical system generated by
the single polynomial $f_0(z)=z^2$.
Note that this discussion does not depend on the parameter $\alpha$.
Consequently, we can compute $T_{\infty,\tilde\tau_0}$ explicitly as
\[
T_{\infty,\tilde\tau_0}(z)
=
\begin{cases}
0 & \text{if } |z|\le 1,\\
1 & \text{if } |z|>1.
\end{cases}
\]

The Fatou set at the state $p=0$ is
\(
F_{0}(S_{\tau,\alpha})
=
\{\, z\in\widehat{\mathbb C} : |z|\neq 1 \,\},
\)
and on each connected component of $F_{0}(S_{\tau,\alpha})$
the function $T_{\infty,\tilde\tau_0}$ is constant.

Similarly, if $p=1$, then $P(p,\{1\})=1$, and hence only the map $f_1$ is admissible.
Again, the dynamics reduces to the deterministic system generated by
$f_1(z)=z^2/2$, independently of $\alpha$.
Hence we have
\[
T_{\infty,\tilde\tau_1}(z)
=
\begin{cases}
0 & \text{if } |z|\le 2,\\
1 & \text{if } |z|>2.
\end{cases}
\]

We next consider interior states.
We fix $p\in(0,1)$.

Independently of the parameter $\alpha$, the Fatou set at the state $p$
can be described explicitly as
\[
F_{p}(S_{\tau,\alpha})
=
\{\, z\in\widehat{\mathbb C} : |z|<1 \,\}
\;\cup\;
\{\, z\in\widehat{\mathbb C} : |z|>2 \,\}.
\]

On this set, the escaping probability can be computed directly.
If $|z|<1$, both maps $f_0$ and $f_1$ preserve the closed unit disk, and the orbit does not escape; hence $T_{\infty,\tilde\tau_p}(z)=0$.
If $|z|>2$, every sequence of the maps $f_0$ and $f_1$ sends $z$ to infinity; hence $T_{\infty,\tilde\tau_p}(z)=1$.
A direct computation also shows that the escaping probability is zero at every point of the unit circle $\{\, z\in\widehat{\mathbb C} : |z|=1 \,\}$.
Summarizing this argument, we obtain
\[
T_{\infty,\tilde\tau_p}(z)
=
\begin{cases}
0, & |z|\le1,\\[4pt]
1, & |z|>2.
\end{cases}
\]

The remaining
nontrivial issue is to determine whether
$T_{\infty,\tilde\tau_p}$ is continuous at points of the set
\(
\{\, z\in\widehat{\mathbb C} : 1<|z|\le2 \,\},
\)
where an explicit computation is more delicate.

We first consider the case $\alpha=0$.
For each fixed interior state $p\in(0,1)$, the restriction of
$S_{\tau,0}$ to the singleton state space $\{p\}$ is an i.i.d.\
random dynamical system and satisfies
Assumption~\ref{ass:standing-cooperation}.
Its kernel Julia set is empty.
Hence, by \cite[Theorem~3.22]{MR2747724}, the function
$T_{\infty,\tilde\tau_p}$ is continuous on $\widehat{\mathbb C}$.
Alternatively, the same conclusion follows from
Corollary~\ref{cor:continuity-statewise-escape}
applied to this singleton restriction.

Next, we consider the case $\alpha\in(0,1)$ and verify that the above continuity fails. 
More precisely, we show that
\(
T_{\infty,\tilde\tau_p}(2)<1,
\)
and hence $T_{\infty,\tilde\tau_p}$ is discontinuous at $z=2$.

We decompose according to the event that the index $0$ is chosen at least once and its complement, namely the event that the index $1$ is chosen at every step.

The point $2\in\widehat{\mathbb C}$ tends to infinity if and only if the index $0$ is selected at least once.
Indeed, if $x_n=1$ for all $n$, then $f_1(2)=2$, and the orbit is identically equal to $2$, hence it does not escape.
If the index $0$ is selected for the first time at some step, then the orbit
moves from $2$ to $4$ at that step. Since
\[
|f_0(z)|>|z|
\quad\text{and}\quad
|f_1(z)|>|z|
\qquad\text{whenever }|z|>2,
\]
every subsequent admissible sequence sends the orbit to infinity.

It follows that
\[
T_{\infty,\tilde\tau_p}(2)
=
\mathbf P_p\bigl(x_n=0 \text{ for some } n\in\mathbb N\bigr)
=
1-\mathbf P_p\bigl(x_n=1 \text{ for all } n\in\mathbb N\bigr).
\]

We compute the latter probability.
We set
\[
p_1=p,
\qquad
p_n=(1-\alpha)p_{n-1}+\alpha \quad (n\ge2).
\]
This recursion reflects the update rule under the assumption that the index $1$ has been chosen at all previous steps.
Solving it gives
\[
p_n=1-(1-\alpha)^{\,
n-1
}(1-p).
\]

Therefore
\[
\mathbf P_p\bigl(x_n=1 \text{ for all } n\in\mathbb N\bigr)
=
\prod_{n=1}^\infty p_n
=
\prod_{n=1}^\infty
\bigl(1-(1-\alpha)^{\,
n-1
}(1-p)\bigr).
\]
Since
\[
\sum_{n=1}^\infty
(1-\alpha)^{\,
n-1
}(1-p)
=
\frac{1-p}{\alpha}
<\infty,
\]
the infinite product is strictly positive. Therefore,
\[
T_{\infty,\tilde\tau_p}(2)
=
1-
\prod_{n=1}^\infty
\bigl(1-(1-\alpha)^{\,
n-1
}(1-p)\bigr)
<1.
\]
Hence $T_{\infty,\tilde\tau_p}$ is discontinuous at $z=2$.

The hypothesis of
Corollary~\ref{cor:continuity-statewise-escape}
is the simultaneous condition
\[
J_{\ker,q}(S_{\tau,\alpha})=\emptyset
\qquad\text{for every }q\in W.
\]
Consequently, the discontinuity found above implies only that this simultaneous
condition fails. It does not imply that the kernel Julia set is nonempty at the
same interior state $p$.

In the present example, the nonempty kernel Julia sets occur at the boundary
states:
\[
J_{\ker,0}(S_{\tau,\alpha})
=
\{z\in\mathbb C:|z|=1\},
\qquad
J_{\ker,1}(S_{\tau,\alpha})
=
\{z\in\mathbb C:|z|=2\}.
\]
On the other hand,
\(
J_{\ker,q}(S_{\tau,\alpha})=\emptyset
\) for every 
\(
q\in(0,1).
\)
Indeed, every finite word in $\{0,1\}$ is admissible from an interior state,
and every state reachable in finitely many steps remains in $(0,1)$.
By the description of the Fatou set above,
\[
J_v(S_{\tau,\alpha})
=
\{z\in\mathbb C:1\le |z|\le2\}
\qquad\text{for every }v\in(0,1).
\]
We fix $q\in(0,1)$ and set
\[
v_{00}:=u(u(q,0),0),
\qquad
v_1:=u(q,1).
\]
Then
\[
f_0^2\in H_q^{v_{00}}(S_{\tau,\alpha}),
\qquad
f_1\in H_q^{v_1}(S_{\tau,\alpha}),
\]
and
\[
(f_0^2)^{-1}\bigl(J_{v_{00}}(S_{\tau,\alpha})\bigr)
=
\{z\in\mathbb C:1\le |z|\le2^{1/4}\},
\]
whereas
\[
f_1^{-1}\bigl(J_{v_1}(S_{\tau,\alpha})\bigr)
=
\{z\in\mathbb C:\sqrt2\le |z|\le2\}.
\]
These two sets are disjoint. Therefore the intersection defining
$J_{\ker,q}(S_{\tau,\alpha})$ is empty.

The positive-probability path on which the index $1$ is selected at every step
is responsible for the failure of escape from $2$, but the existence of this
single path does not imply nonemptiness of the kernel Julia set, which is
defined by an intersection over all admissible finite compositions.

We also note the corresponding statewise filled-in Julia sets:
\[
K_p(S_{\tau,\alpha})
=
\begin{cases}
\{z\in\mathbb C:|z|\le 1\}, & p\in[0,1),\\[4pt]
\{z\in\mathbb C:|z|\le 2\}, & p=1.
\end{cases}
\]

\medskip
\noindent
\textbf{Step 3.} Stationary-Averaged Escaping Behavior.

In this step, we assume that $\alpha\in(0,1)$.
We next consider stationary averaging.
Since both $0$ and $1$ are absorbing states for the induced state chain,
the Dirac measures $\delta_0$ and $\delta_1$ are stationary.
Hence every convex combination of them is stationary.
Conversely, we take $\pi\in\Stat$ and a random variable $R$ with distribution $\pi$.
Conditionally on $R$, we take $I$ to be a Bernoulli random variable with parameter $R$
and set
\[
R':=(1-\alpha)R+\alpha I.
\]
Then
\[
\mathbf E\bigl[(R')^2-R^2\mid R\bigr]
=
\alpha^2R(1-R).
\]
Since $\pi$ is stationary, it follows that
$\mathbf E[R(1-R)]=0$. Hence $\pi$ is supported on $\{0,1\}$.
Together with the observation above,
this gives
\[
\Stat
=
\left\{
\theta\,\delta_0+(1-\theta)\,\delta_1 \in \mathfrak{M}_1(W)
: 0\le \theta\le 1
\right\}.
\]

We set
\(
\pi_{\theta}:=\theta\delta_0+(1-\theta)\delta_1\in\Stat.
\)
Then, by Definition~\ref{def:stationary-averaged-escape},
\[
T_{\infty,\tau,\pi_\theta}(z)
=
\int_W T_{\infty,\tilde\tau_p}(z)\,\pi_\theta(dp)
=
\theta\,T_{\infty,\tilde\tau_0}(z)+(1-\theta)\,T_{\infty,\tilde\tau_1}(z).
\]
Using the explicit formulas for $T_{\infty,\tilde\tau_0}$ and
$T_{\infty,\tilde\tau_1}$, we obtain
\[
T_{\infty,\tau,\pi_\theta}(z)
=
\begin{cases}
0, & |z|\le 1,\\[4pt]
\theta, & 1<|z|\le 2,\\[4pt]
1, & |z|>2.
\end{cases}
\]

In particular, $T_{\infty,\tau,\pi_\theta}$ is discontinuous at $z=2$
for every $\theta\in[0,1)$.
Varying $\theta$ over $[0,1]$, we obtain
\[
\{\,T_{\infty,\tau,\pi}(z):\pi\in\Stat\,\}
=
\begin{cases}
\{0\}, & |z|\le 1,\\[4pt]
[0,1], & 1<|z|\le 2,\\[4pt]
\{1\}, & |z|>2.
\end{cases}
\]
Thus, for each $z$ with \(1<|z|\le 2\), the stationary-averaged escaping probability
ranges over \([0,1]\) as the stationary distribution varies.

We finally note that
Proposition~\ref{prop:recovery-pi-essential}
does not apply to the present example.
Indeed, for every stationary distribution $\pi_\theta$ and every measurable
set $B\subset W$ with $\pi_\theta(B)>0$, the essential filled-in set
$K_B^{\pi_\theta}(S_{\tau,\alpha})$, whenever nonempty, contains
\(
\{z\in\mathbb C:|z|\le1\}.
\)
Hence two such nonempty essential filled-in sets cannot be disjoint.
\end{example}

\begin{example}
\label{ex:two-regimes-truncated}
We then consider a slight modification of the preceding example.

\medskip
\noindent
\textbf{Step 1.} Description of the polynomial RSCC.

We retain $(X,\mathcal X)$, $(f_x)_{x\in X}$, $(\tau_x)_{x\in X}$, and $(\Gamma_x)_{x\in X}$
from Example~\ref{ex:two-regimes-simplex} and modify only the state chain.
We fix $\varepsilon\in(0,\tfrac12)$ and $\alpha\in(0,1)$.
We take
\(
W_\varepsilon:=[\varepsilon,1-\varepsilon] \subset \mathbb{R}
\)
with the Euclidean metric and set $\mathcal W_\varepsilon:=\mathcal B(W_\varepsilon)$.
For $p\in W_\varepsilon$, we define the transition probabilities by
\[
P(p,\{1\}):=p,
\qquad
P(p,\{0\}):=1-p,
\]
and define the update map $u_\varepsilon:W_\varepsilon\times X\to W_\varepsilon$ by
\[
u_\varepsilon(p,x)
=
\begin{cases}
\varepsilon,
& (1-\alpha)p+\alpha x\le \varepsilon,\\[6pt]
(1-\alpha)p+\alpha x,
& \varepsilon < (1-\alpha)p+\alpha x <1-\varepsilon,\\[6pt]
1-\varepsilon,
& (1-\alpha)p+\alpha x\ge 1-\varepsilon.
\end{cases}
\]
Thus, for each $\varepsilon\in(0,\tfrac12)$ and $\alpha\in(0,1)$, we obtain the polynomial RSCC
\[
S_{\tau,\alpha}^{\varepsilon}
=
\{(W_\varepsilon,\mathcal W_\varepsilon),(X,\mathcal X),u_\varepsilon,P,\{\Gamma_x\}_{x\in X}\}
\quad\text{on }\widehat{\mathbb C},
\]
which satisfies Assumption~\ref{ass:standing-cooperation}.
Indeed,
Assumption~\ref{ass:standing-cooperation}\ref{ass:finite-index-space}
and \ref{ass:compact-polynomial-supports}
follow directly from the definitions. Since truncation to
$[\varepsilon,1-\varepsilon]$ is $1$--Lipschitz, for every
$p,q\in W_\varepsilon$ and $a^{(n)}\in X^n$,
\[
\left|pa^{(n)}-qa^{(n)}\right|
\le
(1-\alpha)^n|p-q|.
\]
Thus
Assumption~\ref{ass:standing-cooperation}\ref{ass:update-equicontinuity}
holds. Moreover, the same telescoping estimate as in
Example~\ref{ex:two-regimes-simplex} gives
\[
\sup_{n\in\mathbb N}
\sum_{a^{(n)}\in X^n}
\left|
\mathbf P_p([a^{(n)}])
-
\mathbf P_q([a^{(n)}])
\right|
\le
\frac{2}{\alpha}|p-q|.
\]
Hence
Assumption~\ref{ass:standing-cooperation}\ref{ass:prefix-continuity}
also holds.

Since $p\in[\varepsilon,1-\varepsilon]$, both indices are uniformly admissible:
\[
P(p,\{0\})\ge \varepsilon,
\qquad
P(p,\{1\})\ge \varepsilon
\quad\text{for all }p\in W_\varepsilon.
\]
Thus neither index becomes unavailable at any state, and the absorption at states $0$ and $1$ in Example~\ref{ex:two-regimes-simplex}
is absent.

\medskip
\noindent
\textbf{Step 2.} Statewise Escaping Behavior.

We fix $p\in W_\varepsilon$.
As in Example~\ref{ex:two-regimes-simplex}, the Fatou set at the state $p$ can be
described explicitly by
\[
F_{p}\bigl(S_{\tau,\alpha}^{\varepsilon}\bigr)
=
\{\, z\in\widehat{\mathbb C} : |z|<1 \,\}
\;\cup\;
\{\, z\in\widehat{\mathbb C} : |z|>2 \,\}.
\]
On this set the escaping probability can be computed directly, and we obtain
\[
T_{\infty,\tilde\tau_p}(z)
=
\begin{cases}
0, & |z|\le 1,\\[4pt]
1, & |z|>2.
\end{cases}
\]
This agrees with Lemma~\ref{lem:escape-local-constancy-Fatou}, which asserts that
$T_{\infty,\tilde\tau_p}$ is locally constant on
$F_{p}(S_{\tau,\alpha}^{\varepsilon})$.

The remaining issue is the behavior on the set \(\{\, z\in\widehat{\mathbb C} : 1<|z|\le 2 \,\}\).

In contrast to Example~\ref{ex:two-regimes-simplex}, the absorbing states $0$ and $1$ are excluded,
and both indices remain admissible at every state.
Consequently, every finite word in $\{0,1\}$ is admissible from every
state. The disjoint-preimage argument used for the interior states in
Example~\ref{ex:two-regimes-simplex} therefore applies at every state.
It follows that
\[
J_{\ker,p'}\bigl(S_{\tau,\alpha}^{\varepsilon}\bigr)=\emptyset
\qquad\text{for all }p'\in W_\varepsilon.
\]
Hence Corollary~\ref{cor:continuity-statewise-escape} applies, and we conclude that
\(T_{\infty,\tilde\tau_p}\) is continuous on \(\widehat{\mathbb C}\).
We now verify directly that
\(
T_{\infty,\tilde\tau_p}(2)=1,
\)
which should be compared with Example~\ref{ex:two-regimes-simplex}.

As before, the point \(2\in\widehat{\mathbb C}\) tends to infinity if and only if
the index \(0\) is selected at least once.
Therefore
\[
T_{\infty,\tilde\tau_p}(2)
=
\mathbf P_p\bigl(x_n=0 \text{ for some } n\in\mathbb N\bigr)
=
1-\mathbf P_p\bigl(x_n=1 \text{ for all } n\in\mathbb N\bigr).
\]

In contrast to Example~\ref{ex:two-regimes-simplex}, however, the latter
probability is zero.
Indeed, for every \(q\in W_\varepsilon=[\varepsilon,1-\varepsilon]\), we have
\(
P(q,\{1\})=q\le 1-\varepsilon.
\)
Hence, for every \(n\in\mathbb N\),
\[
\mathbf P_p(x_1=1,\ldots,x_n=1)
\le
(1-\varepsilon)^n.
\]
Letting \(n\to\infty\), we obtain
\(
\mathbf P_p\bigl(x_n=1 \text{ for all } n\in\mathbb N\bigr)=0,
\)
 and thus
\(T_{\infty,\tilde\tau_p}(2)=1\).

We also note that the corresponding statewise filled-in Julia set is independent
of \(p\in W_\varepsilon\) and of \(\alpha\), and is given by
\[
K_p\bigl(S_{\tau,\alpha}^{\varepsilon}\bigr)
=
\{z\in\mathbb C:|z|\le 1\}
\qquad\text{for every }p\in W_\varepsilon.
\]

\medskip
\noindent
\textbf{Step 3.} Stationary-Averaged Escaping Behavior.

We denote by $\Stat_\varepsilon$ the set of stationary distributions of the
induced state chain on $W_\varepsilon$.
Since $W_\varepsilon$ is compact and the induced state chain is Feller, we have
\(
\Stat_\varepsilon\neq\emptyset.
\)
We denote by $Q_\varepsilon$ the Markov transition kernel of the
induced state chain on $W_\varepsilon$.
In fact, $\Stat_\varepsilon$ is a singleton. To prove uniqueness, we choose
$N\in\mathbb N$ such that
\[
(1-\alpha)^N(1-\varepsilon)\le\varepsilon.
\]
If the index $0$ is selected at each of the first $N$ steps, then the state
after these steps is $\varepsilon$, independently of the initial state.
Since $P(q,\{0\})\ge\varepsilon$ for every $q\in W_\varepsilon$, we have
\[
Q_\varepsilon^N(p,\cdot)
\ge
\varepsilon^N\delta_\varepsilon
\qquad
\text{for every }p\in W_\varepsilon.
\]
This Doeblin minorization implies that the stationary distribution is
unique. We therefore write
\(
\Stat_\varepsilon=\{\pi\}.
\)
By Theorem~\ref{thm:continuity-stationary-averaged}, the
function \(T_{\infty,\tau,\pi}\) is continuous on \(\widehat{\mathbb C}\).

We finally note that
Proposition~\ref{prop:recovery-pi-essential}
does not apply to the present example.
Indeed,
\[
K_p(S_{\tau,\alpha}^{\varepsilon})
=
\{z\in\mathbb C:|z|\le1\}
\qquad
\text{for every }p\in W_\varepsilon.
\]
Consequently, all nonempty essential filled-in sets associated with
positive-measure subsets of $W_\varepsilon$ coincide, and the disjointness
condition in Proposition~\ref{prop:recovery-pi-essential} cannot hold.
\end{example}

\begin{example}
\label{ex:recovery-pi-essential-infinite}
We finally present an infinite-state model in which
Proposition~\ref{prop:recovery-pi-essential}
yields nontrivial information without an explicit formula for the escaping
probabilities.

\medskip
\noindent
\textbf{Step 1.} Description of the polynomial RSCC.

We take
\(
W:=[0,1]\times\{0,1\}
\)
and equip it with the metric
\[
d_W\bigl((s,i),(t,j)\bigr):=|s-t|+|i-j|,
\qquad
s,t\in[0,1],\ \ i,j\in\{0,1\}.
\]
Then $W$ is a compact metric space.
We set $\mathcal W:=\mathcal B(W)$.
We set
\(
X:=\{0,1,2,3\}
\)
and
\(
\mathcal X:=\mathcal P(X).
\)

We define
\[
\rho(s):=\frac14+\frac{s}{2},
\qquad s\in[0,1],
\]
so that
\[
\frac14\le \rho(s)\le\frac34
\qquad
\text{for every }s\in[0,1].
\]
We define four polynomials by
\[
f_0(z):=z^2,
\qquad
f_1(z):=\frac{z^2}{2},
\]
\[
f_2(z):=(z-3)^2+3,
\qquad
f_3(z):=\frac{(z-3)^2}{2}+3.
\]

For each $x\in X$, we set
\[
\tau_x:=\delta_{f_x},
\qquad
\Gamma_x:=\supp(\tau_x)=\{f_x\}.
\]

We define the update map $u:W\times X\to W$ by
\[
u((s,i),0):=\left(\frac{s}{2},\,i\right),
\qquad
u((s,i),1):=\left(\frac{s+1}{2},\,i\right),
\]
\[
u((s,i),2):=\left(\frac{s}{2},\,
i\right),
\qquad
u((s,i),3):=\left(\frac{s+1}{2},\,
i\right).
\]
Then $u$ is continuous on all of $W\times X$.

We define the transition probabilities by
\[
P((s,0),\{0\})=\rho(s),
\qquad
P((s,0),\{1\})=1-\rho(s),
\]
\[
P((s,0),\{2\})=P((s,0),\{3\})=0,
\]
and
\[
P((s,1),\{0\})=P((s,1),\{1\})=0.
\]
\[
P((s,1),\{2\})=\rho(s),
\qquad
P((s,1),\{3\})=1-\rho(s).
\]

We set
\[
S_\tau
:=
\{(W,\mathcal W),(X,\mathcal X),u,P,\{\Gamma_x\}_{x\in X}\}.
\]
Then $S_\tau$ satisfies
Assumption~\ref{ass:standing-cooperation}.
Indeed, for every $a^{(n)}\in X^n$,
\[
d_W\bigl((s,i)a^{(n)},(t,j)a^{(n)}\bigr)
=
2^{-n}|s-t|+|i-j|
\le
d_W\bigl((s,i),(t,j)\bigr).
\]
Hence the family of iterated update maps is equicontinuous, and
Assumption~\ref{ass:standing-cooperation}\ref{ass:update-equicontinuity}
holds.

For two states $(s,i)$ and $(t,i)$ in the same component, the
$\ell^1$-distance between their one-step probability vectors is
\[
2|\rho(s)-\rho(t)|=|s-t|.
\]
After $k$ common update steps, the distance between their first coordinates is
$2^{-k}|s-t|$. A telescoping estimate therefore gives
\[
\sum_{a^{(n)}\in X^n}
\left|
\mathbf P_{(s,i)}([a^{(n)}])
-
\mathbf P_{(t,i)}([a^{(n)}])
\right|
\le
\sum_{k=0}^{n-1}2^{-k}|s-t|
\le
2|s-t|.
\]
Since the two components of $W$ are separated by distance $1$, every state has
a neighborhood contained in its component. Thus
Assumption~\ref{ass:standing-cooperation}\ref{ass:prefix-continuity}
holds. The remaining conditions,
Assumption~\ref{ass:standing-cooperation}\ref{ass:finite-index-space}
and \ref{ass:compact-polynomial-supports},
follow directly from the definitions.

Thus, from a state in $[0,1]\times\{0\}$ only the indices $0,1$ are admissible,
whereas from a state in $[0,1]\times\{1\}$ only the indices $2,3$ are admissible.
In particular, the two components
\[
[0,1]\times\{0\}
\quad\text{and}\quad
[0,1]\times\{1\}
\]
are stochastically closed.

\medskip
\noindent
\textbf{Step 2.} Statewise Escaping Behavior.

We set
\[
D_1:=\{z\in\mathbb C:|z|\le1\},
\qquad
D_2:=\{z\in\mathbb C:|z-3|\le1\}.
\]

We fix $w\in[0,1]\times\{0\}$. Since both $\rho(s)$ and $1-\rho(s)$ are bounded
below by $1/4$, every finite word in $\{0,1\}$ is admissible from $w$.
The disk $D_1$ is forward invariant under both maps, since
\[
f_0(D_1)\cup f_1(D_1)\subset D_1.
\]
Thus every admissible finite composition is uniformly bounded on $D_1$.
On the other hand, if $z\notin D_1$, then
$
f_0^n(z)\longrightarrow\infty
$
as $n\to\infty$,
and every finite word consisting only of the symbol $0$ is admissible. Hence
\[
K_w(S_\tau)=D_1
\qquad
\text{for every }w\in[0,1]\times\{0\}.
\]

Similarly, every finite word in $\{2,3\}$ is admissible from each
$w\in[0,1]\times\{1\}$, and
\[
f_2(D_2)\cup f_3(D_2)\subset D_2.
\]
If $z\notin D_2$, then
$
f_2^n(z)\longrightarrow\infty
$
as $n\to\infty$.
Therefore
\[
K_w(S_\tau)=D_2
\qquad
\text{for every }w\in[0,1]\times\{1\}.
\]

Thus the statewise filled-in Julia set depends only on the component of \(W\)
containing the initial state:
\[
K_w(S_\tau)=
\begin{cases}
\{z\in\mathbb C:|z|\le 1\}, & w\in[0,1]\times\{0\},\\[4pt]
\{z\in\mathbb C:|z-3|\le 1\}, & w\in[0,1]\times\{1\}.
\end{cases}
\]

The statewise escaping probabilities are not indicator functions of the
complements of these disks. For example, we take
$w\in[0,1]\times\{0\}$ and suppose that
\[
1<|z|<\sqrt2.
\]
If the symbol $1$ is selected at the first step, then $f_1(z)\in D_1$, and the
orbit never escapes. In contrast, a sufficiently long initial block consisting
only of the symbol $0$ sends $z$ outside the disk $\{z\in\mathbb C:|z|\le2\}$, after which
every admissible orbit tends to infinity. Both cylinder events have positive
probability. Consequently,
\[
0<T_{\infty,\tilde\tau_w}(z)<1.
\]
The analogous conclusion holds on $[0,1]\times\{1\}$ after translation by $3$.
Thus the statewise escaping probabilities are genuinely nontrivial and are not
computed explicitly below.

\medskip
\noindent
\textbf{Step 3.} Stationary-Averaged Escaping Behavior.

The restriction of the induced state chain to each of the two compact
components is Feller and therefore admits a stationary distribution.
We choose stationary distributions $\pi_1$ and $\pi_2$ supported on
$[0,1]\times\{0\}$ and $[0,1]\times\{1\}$, respectively.
Since both components are stochastically closed, we may regard
$\pi_1$ and $\pi_2$ as elements of $\Stat$ for the full state chain.

We set
\[
\pi:=\frac12\pi_1+\frac12\pi_2\in\Stat.
\]

We now verify the hypotheses of
Proposition~\ref{prop:recovery-pi-essential}. We set
\[
B_1:=[0,1]\times\{0\},
\qquad
B_2:=[0,1]\times\{1\}.
\]
Then
\[
\pi(B_1)=\pi(B_2)=\frac12>0.
\]
By Step 2,
\[
K_w(S_\tau)=\{z\in\mathbb C:|z|\le 1\}
\quad\text{for every } w\in B_1,
\]
and
\[
K_w(S_\tau)=\{z\in\mathbb C:|z-3|\le 1\}
\quad\text{for every } w\in B_2.
\]
Consequently,
\[
K_{B_1}^{\pi}(S_\tau)=\{z\in\mathbb C:|z|\le 1\},
\qquad
K_{B_2}^{\pi}(S_\tau)=\{z\in\mathbb C:|z-3|\le 1\}.
\]
These two sets are nonempty and disjoint:
\[
K_{B_1}^{\pi}(S_\tau)\neq\emptyset,\qquad
K_{B_2}^{\pi}(S_\tau)\neq\emptyset,\qquad
K_{B_1}^{\pi}(S_\tau)\cap K_{B_2}^{\pi}(S_\tau)=\emptyset.
\]

Hence conditions~\ref{itm:ess-posmass},
\ref{itm:ess-commonK1}, \ref{itm:ess-commonK2}, and
\ref{itm:ess-disjoint} of
Proposition~\ref{prop:recovery-pi-essential}
are satisfied. The proposition therefore gives
\(
T_{\infty,\tau,\pi}(z)>0
\)
for every 
\(
z\in\widehat{\mathbb C}.
\)
Moreover, since
\(
0\in K_{B_1}^{\pi}(S_\tau),
\)
the estimate in the proof of the proposition gives
\[
T_{\infty,\tau,\pi}(0)
\le
1-\pi(B_1)
=
\frac12
<1.
\]
Consequently,
\(
0<T_{\infty,\tau,\pi}(0)<1.
\)
Thus Proposition~\ref{prop:recovery-pi-essential}
establishes the nondegeneracy of the stationary-averaged escaping probability
without requiring an explicit formula for it.
\end{example}

\medskip
\noindent
\textbf{Acknowledgments.}\enspace
The author would like to thank Prof. Johannes Jaerisch and Prof. Hiroki Sumi
for their valuable comments and helpful suggestions.
The author is also grateful to the anonymous reviewers for their careful
reading and constructive comments.

\bibliographystyle{alpha}
\bibliography{references}
\end{document}